\documentclass[11pt]{article}

\usepackage{amsmath,amssymb,amsthm,amsfonts}
\usepackage{mathtools}
\usepackage{mathrsfs}
\usepackage{bm}
\usepackage{graphicx}
\usepackage{float}
\usepackage{multirow}
\usepackage{booktabs}
\usepackage{longtable}
\usepackage{arydshln}
\usepackage{marginnote}
\usepackage{pdfcomment}
\usepackage{enumerate}
\usepackage[colorinlistoftodos]{todonotes}
\usepackage{cite}
\usepackage{hyperref}
\usepackage[a4paper,margin=2.5cm]{geometry}

\hypersetup{
	colorlinks=true,
	linkcolor=blue,
	citecolor=blue,
	urlcolor=blue
}

\allowdisplaybreaks[2]

\numberwithin{equation}{section}

\newtheorem{theorem}{Theorem}[section]
\newtheorem{lemma}[theorem]{Lemma}
\newtheorem{proposition}[theorem]{Proposition}
\newtheorem{corollary}[theorem]{Corollary}
\newtheorem{definition}[theorem]{Definition}
\newtheorem{example}[theorem]{Example}
\newtheorem{remark}[theorem]{Remark}

\newcommand{\HH}{\mathbb{H}}

\newcommand{\SR}{\mathcal{SR}}

\title{\Large\bf
	Slice Regular Composition Operators on Quaternionic Fock Spaces via Matrix Realization
	\footnotetext{\hspace{-0.35cm}
		\endgraf{\it E-mail: zhaopeng.lin@mail.dlut.edu.cn (Zhaopeng Lin)}
		\endgraf \hspace{1.1cm} {\it lyfdlut@dlut.edu.cn (Yufeng Lu)}
		\endgraf \hspace{1.1cm} {\it zuchao@dlut.edu.cn (Chao Zu)}
		\endgraf Y. Lu was supported by the National Natural Science Foundation of China
		(Grant No. 12031002). C. Zu was supported by the National Natural Science
		Foundation of China (Grant No. 12401151), and the Postdoctoral Researcher
		Foundation of China (Grant No. GZB20240100).
}}

\author{
	Zhaopeng Lin,
	Yufeng Lu,
	Chao Zu\thanks{Corresponding author}
}
\date{}

\begin{document}
	
	\maketitle
	
	\vspace{-0.8cm}
	
	\begin{center}
		\begin{minipage}{14cm}\small
{\noindent{\bf Abstract} \quad
We characterize the boundedness and compactness of slice regular
composition operators between quaternionic Fock spaces for the full range \(0<p,q<\infty\), without assuming that the composition symbol preserves a fixed complex slice.  As applications of the same method, we also obtain corresponding criteria for weighted composition operators and for products of Volterra-type integral operators with slice regular composition operators. The main tool is a fixed-slice matrix realization of the regular product, which represents slice regular composition on a fixed complex slice through a holomorphic \(2\times 2\) matrix functional calculus. This representation reveals a genuinely quaternionic rigidity phenomenon: boundedness imposes affine restrictions on the eigenvalue functions of the associated matrix symbol rather than on the original symbol itself. In particular, the original symbol need not be affine, and affine eigenvalue functions alone do not characterize boundedness.
	\endgraf
	\endgraf
	}
		\end{minipage}
	\end{center}

	\section{Introduction} 
Composition operators form one of the fundamental classes of linear
operators on spaces of analytic functions. If \(\varphi\) is an analytic self-map of the underlying domain, the associated composition operator is
defined by
\[
C_\varphi f=f\circ\varphi.
\]
A central problem in the theory is to determine how operator-theoretic properties of \(C_\varphi\), such as boundedness, compactness, and essential norm estimates, reflect the function-theoretic and geometric properties of the symbol \(\varphi\). In this way, composition operators provide a natural link between operator theory and analytic function theory; see
\cite{MR1237406,MR1397026}.

The nature of this problem depends strongly on the underlying function 
space. On the classical Hardy and Bergman spaces over the unit disk, every
analytic self-map \(\varphi\) induces a bounded composition operator
\(C_\varphi\), by Littlewood's subordination principle. Hence the main
operator-theoretic questions in these spaces concern compactness, essential
norm estimates, Schatten-class properties, and their dependence on the
boundary behavior of \(\varphi\). In the Hardy-space theory, angular
derivatives and Nevanlinna counting functions play a central role, whereas on Bergman spaces compactness is naturally described in terms of the vanishing Carleson behavior of the pull-back measure; see
\cite{MR1575208,MR192062,MR223914,MR854144,MR881273}.

The situation on Fock spaces is quite different. Because of the Gaussian weight, holomorphic self-maps no 
longer automatically induce bounded composition operators. Instead, the 
behavior of the symbol at infinity is severely restricted. This leads to 
the  affine rigidity phenomenon. For complex Fock spaces, 
Carswell, MacCluer and Schuster determined the holomorphic symbols  inducing 
bounded and compact composition operators \cite{MR2034214}. In one complex variable, boundedness forces $
\varphi(z)=az+b$ with $ |a|<1$, 
or $
\varphi(z)=az$ with $|a|=1$, 
while compactness is equivalent to the first alternative.   

Weighted composition operators provide a natural extension. Given a weight 
\(u\) and a holomorphic self-map \(\varphi\),  one defines
\[
W_{u,\varphi}f=u\cdot(f\circ\varphi).
\]
They include both composition and multiplication operators as special 
cases, and they arise naturally in several problems  of  geometric function theory; for instance,
Brennan's conjecture is related to questions about certain weighted composition operators
\cite{MR2198381,MR2886614}.Weighted composition operators have been studied extensively on Hardy and 
Bergman spaces; see, for example, \cite{MR2078907,MR2342670}. On Fock spaces, their boundedness and 
compactness are typically characterized by Berezin-type testing quantities, 
or equivalently by Fock--Carleson measure conditions; see, for example, \cite{ueki2007weighted,MR3905527}. Special operator 
classes of weighted composition operators on Fock spaces, such as normal and
isometric operators, have also been studied; see \cite{MR3239622}.
The same approach also applies to products of Volterra-type integral
operators and composition operators. By means of a Littlewood--Paley type
norm equivalence, the boundedness and compactness of such products can be studied through estimates for weighted composition operators; see, for example, \cite{MR2957216, MR3192295}.

Quaternionic analysis provides a noncommutative analogue of one-variable
complex analysis. Let  \(\mathbb S\) be the sphere of imaginary units in
\(\mathbb H\), and, for each \(I\in\mathbb S\), set  \[
\mathbb C_I:=\{x+yI:\ x,y\in\mathbb R\}.
\]  Each \(\mathbb C_I\) is called a complex slice. The modern theory of slice regular functions, initiated
by Cullen \cite{MR173012} and developed further by Gentili and Struppa
\cite{MR2353257}, retains several basic features of complex function theory,
including power series expansions and Cauchy formulas. This theory has led to quaternionic analogues of classical spaces
of analytic functions, including Hardy spaces
\cite{MR3801294,MR3358083}, Bergman spaces \cite{MR34064751}, and Fock
spaces \cite{MR3587897}. On each fixed slice, 
these spaces are closely related to their complex counterparts, but their 
operator theory is essentially affected by the noncommutativity of 
quaternionic multiplication.

Composition operators in the quaternionic setting have been studied extensively. Ren and Wang introduced slice regular composition operators
and established, among other results, a Littlewood-type subordination
principle and dynamical consequences \cite{MR3482788}. On quaternionic Fock
spaces, Lian and Liang obtained boundedness and compactness criteria for
weighted composition operators \cite{MR4162404}. Further work has treated
differences of weighted composition operators and special classes, such as
self-adjoint, unitary, and isometric weighted 
composition operators; see, for example,
\cite{MR4384577,MR4519267,MR4544164}.

A common simplifying assumption in quaternionic composition operator
problems is that the inducing symbol preserves a fixed complex slice.
 More precisely, one assumes that there exists \(I\in\mathbb S\) such that
\[
\psi(\mathbb C_I)\subset\mathbb C_I.
\]
Under this hypothesis, slice regular composition reduces on
\(\mathbb C_I\) to ordinary holomorphic composition. Consequently, the
  problem can be treated by complex methods.

The purpose of the present paper is to remove this fixed-slice-preserving assumption. We consider arbitrary slice regular symbols
\(\psi\in\mathcal{SR}(\mathbb H)\). For such a symbol, the slice regular
composition of
\[
f(\xi)=\sum_{n=0}^{\infty}\xi^n a_n
\]
with \(\psi\) is defined by
\[
f\bullet\psi
=
\sum_{n=0}^{\infty}\psi^{\star n}a_n,
\]
where \(\star\) is the  regular product. However,  \((f\bullet\psi)_I\) cannot, in general, be expressed through
ordinary scalar compositions. This is the main obstruction to a direct application of the classical pull-back measure method.

We treat in parallel three related classes of operators. The first is the slice regular
composition operator
\[
C_\psi^\bullet f=f\bullet\psi .
\]
The second is the weighted composition operator
\[
W_{(u,\psi)}f=u\star(f\bullet\psi),
\qquad u,\psi\in\mathcal{SR}(\mathbb H).
\]
The third consists of products of Volterra-type integral operators and slice regular composition operators: 
\[
(V_{(g,\psi)}f)(z) =
\int_0^z
\bigl(g'\star(f\bullet\psi)\bigr)_I(w)\,\mathrm dw,
\qquad z\in\mathbb C_I,\quad I\in\mathbb S .
\]
Here \(g'\) denotes the slice derivative.  

Our first results give Berezin-type testing criteria for slice regular 
composition operators. Throughout the paper, \(\alpha>0\) is fixed. For
\(I\in\mathbb S\), let \(\mathrm d m_{2,I}\) denote the planar Lebesgue measure
on \(\mathbb C_I\), and let \(F_\alpha^p(\mathbb H)\) denote the quaternionic
Fock space. For
\(w\in\mathbb C_I\), let \(k_w\) denote the normalized reproducing kernel of
\(F_\alpha^2(\mathbb H)\). Define
\begin{equation}\label{eq:composition-testing-quantity}
	\mathcal B_{\psi,I}(w)
	:=
	\int_{\mathbb C_I}
	\left|(k_w\bullet\psi)_I(z)\right|^q
	\mathrm e^{-\frac{\alpha q}{2}|z|^2}
	\,\mathrm d m_{2,I}(z).
\end{equation}

\begin{theorem}\label{thm:composition-p-le-q}
	Let \(\psi\in\mathcal{SR}(\mathbb H)\) and \(0<p\le q<\infty\). Then the
	following assertions hold.
	
	\smallskip
	\noindent\textup{(i)}
	The operator $
	C_\psi^\bullet:
	F_\alpha^p(\mathbb H)\longrightarrow F_\alpha^q(\mathbb H)$ 
	is bounded if and only if, for some, equivalently for every,
	\(I\in\mathbb S\),
	\[
	\sup_{w\in\mathbb C_I}\mathcal B_{\psi,I}(w)<\infty.
	\]
	Moreover,
	\[
	\|C_\psi^\bullet\|
	\approx
	\sup_{w\in\mathbb C_I}\mathcal B_{\psi,I}(w)^{1/q} .
	\]
	
	\smallskip
	\noindent\textup{(ii)}
	The operator $
	C_\psi^\bullet:
	F_\alpha^p(\mathbb H)\longrightarrow F_\alpha^q(\mathbb H)$  is compact if and only if, for some,
	equivalently for every, \(I\in\mathbb S\),
	\[
	\mathcal B_{\psi,I}(w)\longrightarrow0
	\qquad
	\text{as }|w|\to\infty,\quad w\in\mathbb C_I.
	\]
\end{theorem}

\begin{theorem}\label{thm:composition-q-less-p}
	Let \(\psi\in\mathcal{SR}(\mathbb H)\) and \(0<q<p<\infty\). Put
	\(s=p/(p-q)\). Then the following assertions are equivalent:
	\begin{enumerate}[(i)]
		\item \(C_\psi^\bullet:F_\alpha^p(\mathbb H)\to F_\alpha^q(\mathbb H)\)
		is bounded;
		\item \(C_\psi^\bullet:F_\alpha^p(\mathbb H)\to F_\alpha^q(\mathbb H)\)
		is compact;
		\item For some, equivalently for every, \(I\in\mathbb S\),
		\[
		\mathcal B_{\psi,I}\in L^s(\mathbb C_I,\mathrm d m_{2,I}).
		\]
	\end{enumerate}
	Moreover,
	\[
	\|C_\psi^\bullet\|
	\approx
	\left(
	\int_{\mathbb C_I}
	\mathcal B_{\psi,I}(w)^s\,\mathrm d m_{2,I}(w)
	\right)^{1/(qs)}.
	\]
\end{theorem}

The same method gives analogous criteria for weighted composition
operators. More precisely, for
\[
W_{(u,\psi)}:F_\alpha^p(\mathbb H)\to F_\alpha^q(\mathbb H),
\qquad 0<p,q<\infty,
\]
boundedness and compactness are characterized in
Theorems~\ref{thm:weighted-p-le-q} and~\ref{thm:weighted-q-less-p} by the
corresponding weighted Berezin-type testing quantity.

As a further application, the slice Littlewood--Paley norm equivalence yields
the corresponding results for products of Volterra-type integral operators
and slice regular composition operators. More precisely, for
\[
V_{(g,\psi)}:F_\alpha^p(\mathbb H)\to F_\alpha^q(\mathbb H),
\qquad 0<p,q<\infty,
\]
boundedness and compactness are characterized in
Theorems~\ref{thm:volterra-p-le-q} and~\ref{thm:volterra-q-less-p} by the
  testing quantity in \eqref{eq:volterra-testing-quantity}.
In particular, when \(\psi\) is the identity symbol, these results reduce to criteria for boundedness and compactness of the Volterra-type integral operator \(V_g\).
 
The main difficulty is to obtain a fixed-slice representation of
\(f\bullet\psi\) for general \(\psi\in \SR (\mathbb{H})\). Our approach is
based on a matrix realization of the $\star$-product. Fix \(I,J\in\mathbb S\)
with \(J\perp I\). If the splitting of \(f\)  is \(f_I=F+GJ\), define
\[
\Phi_I(f)
:=
\begin{pmatrix}
	F&G\\[2pt]
	-\,G^\#&F^\#
\end{pmatrix},
\qquad
D^\#(z):=\overline{D(\overline z)} .
\]
Then
\[
\Phi_I(f\star h)=\Phi_I(f)\Phi_I(h),
\qquad f,h\in\mathcal{SR}(\mathbb H).
\]
Thus the noncommutative slice product is transformed, on a fixed slice, 
into ordinary multiplication of holomorphic \(2\times2\) matrix-valued 
functions.

This matrix viewpoint gives a concrete interpretation of slice regular 
composition. The slice regular composition \(f\bullet\psi\) can be represented on \(\mathbb C_I\) by
applying the holomorphic functions arising from the splitting of \(f\) to the
matrix symbol \(\Phi_I(\psi)\). In this way, the noncommutative composition problem becomes
a holomorphic matrix functional-calculus problem on the fixed slice. The subsequent analysis is governed by the local spectral structure of
\( \Phi_I(\psi)\).

The matrix representation also clarifies the form of the Fock-space affine 
rigidity in the quaternionic setting. In contrast with the scalar complex case, 
the Berezin-type testing conditions do not directly force the slice regular 
symbol \(\psi\) itself to be affine. Rather, they first impose restrictions on the spectral invariants of 
the fixed-slice matrix symbol \(\Phi_I(\psi)\): the trace $\tau$ must be affine, and the discriminant $\Delta$ must have degree at most two. Consequently, the  
eigenvalue functions have affine asymptotics on unbounded regions; 
in the zero-free-discriminant case, the eigenvalue functions are globally affine. We formulate this spectral rigidity as follows. 

For a set \(E\subset\mathbb C_I\), write  $
E^\#:=\{\overline z:\ z\in E\}$. 
\begin{theorem}
	\label{thm:composition-spectral-rigidity}
Let \(0<p,q<\infty\). Fix \(I,J\in\mathbb S\) with \(J\perp I\), and write the
splitting of \(\psi\) on \(\mathbb C_I\) as
\[
\psi_I=P+QJ,
\qquad P,Q:\mathbb C_I\to\mathbb C_I \text{ entire},
\]
with \(Q\not\equiv0\).  Suppose that either  \begin{equation}\label{eq:rigidity-sup-testing-condition}
	\sup_{w\in\mathbb C_I}\mathcal B_{\psi,I}(w)<\infty,
\end{equation}
or \(0<q<p\) and, with \(s=p/(p-q)\),
\begin{equation}\label{eq:rigidity-Ls-testing-condition}
	\int_{\mathbb C_I}
	\mathcal B_{\psi,I}(w)^s\,\mathrm d m_{2,I}(w)<\infty .
\end{equation}
	Then the following assertions hold. 
	
	\smallskip
	\noindent\textup{(i)}
	The discriminant is a polynomial of degree at most two:
	\[
	\Delta(z)=d_2z^2+d_1z+d_0,
	\qquad d_0,d_1,d_2\in\mathbb R.
	\]
	Moreover, either \(\Delta\) is a nonzero constant, or \(d_2<0\).
	
	\smallskip
	\noindent\textup{(ii)}
	The trace function $\tau$ is affine:
	\[
	\tau(z)=t_1z+t_0,
	\qquad t_0,t_1\in\mathbb R.
	\]
	
	\smallskip
	\noindent\textup{(iii)}
	If \(\Delta\) is nonconstant, then there exists an unbounded simply connected 
  region $
	U $ satisfying $U=
U^\# $
	on which one may choose a holomorphic function \(\rho \) on $U$ with \(\rho^2=\Delta\),
	so that
	\begin{equation}\label{eq:composition-branch-affine-asymptotics}
		\lambda_\pm(z)
		:=\frac{\tau\pm\rho}{2} =
		a_\pm z+b_\pm+O(1/z),
		\qquad z\to\infty,
	\end{equation}
	where
	\[
	a_\pm,b_\pm\in\mathbb C_I,
	\qquad
	0<|a_\pm|\le1,
	\qquad
	|a_+|=|a_-|.
	\]
	
	\smallskip
	\noindent\textup{(iv)}
	If \(\Delta\) has no zeros on \(\mathbb C_I\), then
	\(\Delta\equiv\Delta_0\ne0\) is constant. After choosing a constant square
	root \(\rho\) of \(\Delta_0\), the global eigenvalue functions are affine:
	\[
	\lambda_\pm(z)=az+b_\pm,
	\qquad
	a,b_\pm\in\mathbb C_I,
	\qquad
	|a|\le1.
	\]
\end{theorem}

In the zero-free-discriminant case, Corollary~\ref{cor:composition-zero-free-discriminant}  gives a more explicit fixed-slice form for the symbol $\psi$. 
In particular, the symbol $\psi$ need not be affine even when the associated 
global eigenvalue functions are affine. Moreover, the spectral affine behavior in Theorem~\ref{thm:composition-spectral-rigidity} is only a necessary consequence of the testing conditions and is not sufficient for boundedness. Example~\ref{ex:affine-branches-not-sufficient} shows that two
slice regular symbols may have the same globally affine eigenvalue functions,
with strictly contractive slopes, while one induces a compact slice regular
composition operator and the other induces an unbounded one. Thus boundedness
cannot be characterized by the eigenvalue functions alone; the spectral
coefficients also play an essential role.

The main contributions of the paper are as follows.
\begin{enumerate}
	\item We develop a concrete fixed-slice matrix-functional-calculus representation of slice regular composition for arbitrary slice regular symbols.
	\item We characterize the boundedness and compactness of
	\(C_\psi^\bullet\), \(W_{(u,\psi)}\), and \(V_{(g,\psi)}\) between
	quaternionic Fock spaces in the full range \(0<p,q<\infty\), in terms of
	slice Berezin-type testing quantities.
\item We identify a genuinely quaternionic spectral rigidity phenomenon:
under either testing condition \eqref{eq:rigidity-sup-testing-condition} or
\eqref{eq:rigidity-Ls-testing-condition}, the trace, discriminant, and
eigenvalue functions of the fixed-slice matrix symbol \(\Phi_I(\psi)\) satisfy
the restrictions described in Theorem~\ref{thm:composition-spectral-rigidity}.
In particular, the eigenvalue functions exhibit affine behavior, while the
slice regular symbol \(\psi\) itself need not be affine.
\end{enumerate}

The paper is organized as follows. Section~\ref{sec:preliminaries} recalls
the necessary background on slice regular functions, quaternionic Fock
spaces, the \(\star\)-product and slice regular composition.
Section~\ref{sec:composition-operators} proves the boundedness and compactness criteria for slice regular composition operators and analyzes the spectral rigidity described above.
Section~\ref{sec:applications} contains the applications to weighted composition operators and products of Volterra-type integral operators and slice regular composition operators.  

Throughout the paper, \(A\lesssim B\) means that \(A\le CB\) for a positive
constant \(C\) independent of the test function and of the free variables
under consideration. Unless explicitly stated otherwise, the implicit
constants may depend on \(\alpha,p,q\), the fixed slice. We write \(A\approx B\) if both
\(A\lesssim B\) and \(B\lesssim A\) hold. When \(p<1\) or \(q<1\), the
notation \(\|T\|\) denotes the corresponding operator bound between
quasi-Banach  spaces. 
	\section{Preliminaries}\label{sec:preliminaries}
	
	\subsection{Quaternions and slice regular functions}
	
	Recall that the quaternionic field is denoted by
	\[
	\mathbb{H}
	=
	\{x_0+x_1\mathrm i+x_2\mathrm j+x_3\mathrm k:
	x_l\in\mathbb{R},\ 0\le l\le 3\},
	\]
	which is a four-dimensional noncommutative algebra generated by the imaginary units
	\(\mathrm i,\mathrm j,\mathrm k\), subject to the relations
	\[
	\mathrm i^2=\mathrm j^2=\mathrm k^2=-1,\qquad
	\mathrm i\mathrm j=-\mathrm j\mathrm i=\mathrm k,\qquad
	\mathrm j\mathrm k=-\mathrm k\mathrm j=\mathrm i,\qquad
	\mathrm k\mathrm i=-\mathrm i\mathrm k=\mathrm j.
	\] 
	Let
	\[
	\mathbb S:=\{I\in\HH:\ I^2=-1\},
	\]
	and, for \(I\in\mathbb S\),  set
	\[
	\mathbb C_I:=\{x+Iy:\ x,y\in\mathbb R\}.
	\]
	We write \(q=\operatorname{Re}(q)+\operatorname{Im}(q)\), \(\overline q\) for the quaternionic conjugate, and \(|q|=\sqrt{q\overline q}\) for the Euclidean norm.

	\begin{definition}\cite[Definition 2.1]{colombo2016entire}\label{def:slice-regular}
		Let \(f:\HH\to\HH\) be real differentiable. We say that \(f\) is
		\emph{slice regular} if for every \(I\in\mathbb S\), its restriction $
		f_I:=f|_{\mathbb C_I}$ 
		is holomorphic on \(\mathbb C_I\), that is,
		\[
		\overline{\partial_I}f(x+yI)
		:=
		\frac12\left(
		\frac{\partial}{\partial x}
		+
		I\frac{\partial}{\partial y}
		\right)f_I(x+yI)
		=0
		\]
		for all \(x+yI\in\mathbb C_I\).
		We denote by $
		\mathcal{SR}(\HH)$ the class of slice regular functions on \(\HH\).
	\end{definition}
	
	\begin{definition}\cite[Definition~2.2]{colombo2016entire}
		\label{def:slice-derivative}
		Let $f\in\mathcal{SR}(\mathbb H)$. The \emph{slice derivative}, denoted by $f'$, is defined by
		\[
		f'(z)
		:=
		\begin{cases}
			\partial_I f(z), & z=x+yI,\ y\neq0,\\[4pt]
			\dfrac{\partial f}{\partial x}(x), & z=x\in\mathbb R,
		\end{cases}
		\]
		where
		\[
		\partial_I f(x+yI)
		:=
		\frac12
		\left(
		\frac{\partial}{\partial x}
		-
		I\frac{\partial}{\partial y}
		\right)f_I(x+yI).
		\]
		Moreover, on each slice \(\mathbb C_I\), the operators \(\partial_I\) and 
		\(\overline{\partial_I}\) commute. Then $f'$ is also slice regular.
	\end{definition}

	\begin{lemma}[Splitting Lemma]\cite[Lemma 2.1]{colombo2016entire}\label{lem:splitting}
		Let $
		f\in\mathcal{SR}(\HH)$. 
		For every pair \(I,J\in\mathbb S\) with \(I\perp J\), there exist holomorphic functions $
		F,G:\mathbb C_I\to\mathbb C_I$ 
		such that
		\[
		f_I(z)=F(z)+G(z)J,
		\qquad z\in\mathbb C_I.
		\]
	\end{lemma}
	Throughout the paper, the prime notation is used in the following compatible
	sense. If \(h\in\SR(\HH)\), then \(h'\) denotes the slice derivative of \(h\). If
	\(H\) is a holomorphic \(\mathbb C_I\)-valued function on \(\mathbb C_I\), then
	\(H'\) denotes its usual complex derivative. If \(h_I=H+KJ\) is the splitting of \(h\), then
	\[
	(h')_I=H'+K'J,
	\]
	where the prime on the left denotes the slice derivative and the primes on the
	right denote the usual complex derivatives on \(\mathbb C_I\).

	\begin{proposition}[Representation Formula]\cite[Theorem 2.4]{colombo2016entire}\label{prop:representation}
		Let $
		f\in\mathcal{SR}(\HH)$. 
		For any \(J\in\mathbb S\) and any \(x\pm yJ\in\mathbb C_J\), the following identity holds for every \(q=x+yI\in\HH\):
		\[
		f(x+yI)
		=
		\frac12
		\Big[
		(1-IJ)f(x+yJ)
		+
		(1+IJ)f(x-yJ)
		\Big].
		\]
	\end{proposition}
	
	\subsection{Quaternionic Fock spaces}
	
	Let \(\alpha>0\) and \(0<p<\infty\). For \(I\in\mathbb S\), let $
	\mathrm d m_{2,I}$ 
	denote the Lebesgue measure on \(\mathbb C_I\).
	
	\begin{definition}\label{def:fock-space}\cite[Definition~3.6]{MR3587897}
		Fix \(I\in\mathbb S\). The quaternionic Fock space $
		F_\alpha^p(\HH)$ 
		consists of all $
		f\in\mathcal{SR}(\HH)$ 
		such that
		\[
		\|f\|_{F_\alpha^p(\mathbb H)}^p
		:=
		\frac{\alpha p}{2\pi}
		\int_{\mathbb C_I}
		|f_I(z)|^p
		\mathrm{e}^{-\frac{\alpha p}{2}|z|^2}
		\,\mathrm d m_{2,I}(z)
		<
		\infty.
		\]
	\end{definition}
	
	The resulting space does not depend on the choice of $I$. More precisely, different choices of the slice give equivalent quasi-norms, with constants independent of the function; see \cite{MR3587897}. We shall therefore suppress the dependence on $I$ in the notation. 
	In the Hilbertian case $
	p=2$, 
	the space $
	F_\alpha^2(\HH)$ 
	is a reproducing kernel Hilbert space with kernel
	\[
	K_w(z)
	:=
	\mathrm{e}_\star^{\alpha z\overline w}
	=
	\sum_{n=0}^\infty
	\frac{\alpha^n z^n \overline w^{\,n}}{n!}.
	\]
	The normalized reproducing kernel is $
	k_w(z)
	=
	\mathrm{e}^{-\frac{\alpha}{2}|w|^2}K_w(z)$. 
	When $
	z,w\in\mathbb C_I$, 
	this reduces to the  reproducing kernel on complex Fock space.
	
	For later use,  we denote by
	$
	F_\alpha^p(\mathbb C_I)$ 
	the complex Fock space on \(\mathbb C_I\), that is, the space of all entire functions $
	F:\mathbb C_I\to\mathbb C_I$ 
	such that
	\[
	\|F\|_{F_\alpha^p(\mathbb C_I)}^p
	:=
	\frac{\alpha p}{2\pi}
	\int_{\mathbb C_I}
	|F(z)|^p
	\mathrm{e}^{-\frac{\alpha p}{2}|z|^2}\,
	\mathrm d m_{2,I}(z)
	<
	\infty.
	\]
	For \(z\in\mathbb C_I\) and \(r>0\), we write
\[
D_I(z,r):=\{w\in\mathbb C_I:\ |w-z|<r\}.
\]

	We shall use the following elementary consequence of the splitting lemma. If
	$f_I=F+GJ$, where $J\perp I$, then
	\[
	|f_I(z)|^p
	\approx
	|F(z)|^p+|G(z)|^p,
	\qquad z\in\mathbb C_I.
	\]
	Consequently,
	\begin{equation}\label{eq:splitting-fock-norm-comparison}
		\|f\|_{F_\alpha^p(\mathbb H)}^p
		\approx
		\|F\|_{F_\alpha^p(\mathbb C_I)}^p
		+
		\|G\|_{F_\alpha^p(\mathbb C_I)}^p .
	\end{equation}
	The constants depend only on $p$.

	\subsection{$\star$-product, slice regular composition, and $C^\bullet_\psi $}
	Since the pointwise product of slice regular functions is not, in general, slice regular, one replaces it by the $\star$-product. We first recall its definition in terms of power series.

	\begin{definition}\cite[Definition~1.36]{MR3013643}\label{def:slice-product}
		Let \(f,g\in\mathcal{SR}(\mathbb H)\), and write
		\[
		f(\xi)=\sum_{n=0}^\infty \xi^n a_n,
		\qquad
		g(\xi)=\sum_{n=0}^\infty \xi^n b_n.
		\]
		 The  \(\star\)-product of \(f\)
		and \(g\) is defined by
		\[
		(f\star g)(\xi)
		:=
		\sum_{n=0}^\infty \xi^n
		\left(
		\sum_{k=0}^n a_k b_{n-k}
		\right).
		\]
	\end{definition}
	
	The preceding definition admits the following fixed-slice representation; see 
	\cite[Definition~1.37]{MR3013643}. Fix \(I,J\in\mathbb S\) with 
	\(J\perp I\). If
	\[
	f_I =F +G J,
	\qquad
	g_I =H +K J,
	\]
	where $
	F,G,H,K:\mathbb C_I\to\mathbb C_I$ 
	are holomorphic. 	For a holomorphic function   \(D:\mathbb C_I\to\mathbb C_I\), we write
	\[
	D^\#(z):=\overline{D(\overline z)},
	\qquad z\in\mathbb C_I.
	\] 
	Then the restriction of \(f\star g\) to \(\mathbb C_I\) is given by
	\[
	(f\star g)_I 
	=
	\bigl(F H -G K^\# \bigr)
	+
	\bigl(F K +G H^\# \bigr)J.
	\] 
	\begin{definition}\label{def:slice-composition}\cite[Definition~4.1]{MR3482788}
		Let \(f,\psi\in\mathcal{SR}(\mathbb H)\), and write
		\[
		f(\xi)=\sum_{n=0}^{\infty} \xi^n a_n,
		\qquad \xi\in\mathbb H .
		\]
		The slice regular composition of $f$ with $\psi$ is defined by
		\[
		(f\bullet \psi)(\xi)
		:=
		\sum_{n=0}^\infty \psi^{\star n}(\xi)a_n,
		\]
		where $\psi^{\star n}$ denotes the $n$-th power of $\psi$ with respect to the
		$\star$-product, with $\psi^{\star0}\equiv1$.
	\end{definition}
	
	This is the right-coefficient convention: the coefficients $a_n$ occur
	on the right of the powers $\psi^{\star n}$. Since $f$ is entire slice
	regular, the series converges locally uniformly and defines a slice
	regular function; see
	\cite[Remark~4.2 and Proposition~4.12]{MR3482788}.
	
	Unlike the $\star$-product, slice regular composition does not admit, for
	a general symbol $\psi$, a direct scalar fixed-slice formula in terms of
	the splitting components of $f$ and $\psi$.     
If $\psi$ preserves the fixed complex slice $\mathbb C_I$, then the restriction of
slice regular composition to $\mathbb C_I$ reduces to ordinary
composition:
	\[
	(f\bullet\psi)_I
	=
	(F\circ P)+(G\circ P)J.
	\]
	In particular, if $\psi$ is intrinsic, this reduction holds
	simultaneously on every slice, and $f\bullet\psi$ agrees with the usual
	composition $f\circ\psi$ on $\mathbb H$.
	
 \begin{definition}
 	\label{def:composition-weighted-composition}
 	Let \(u,\psi\in\mathcal{SR}(\mathbb H)\). The slice regular composition
 	operator induced by \(\psi\) is defined by
 	\[
 	C^\bullet_\psi f:=f\bullet\psi,
 	\qquad f\in\mathcal{SR}(\mathbb H).
 	\]
 	The weighted slice regular composition operator induced by \(u\) and
 	\(\psi\) is defined by
 	\[
 	W_{(u,\psi)}f
 	:=
 	u\star(f\bullet\psi),
 	\qquad f\in\mathcal{SR}(\mathbb H).
 	\]
 	In particular, \(C^\bullet_\psi=W_{(1,\psi)}\).
 \end{definition}
 
 \begin{definition}
 	\label{def:volterra-composition}
 	Let \(g,\psi\in\mathcal{SR}(\mathbb H)\). The Volterra-type integral
 	operator induced by \(g\) is defined by
 	\[
 	(V _g f)(z)
 	:=
 	\int_0^z (g'\star f)_I(w)\,\mathrm dw,
 	\qquad z\in\mathbb C_I,\quad I\in\mathbb S,
 	\]
 	where the integral is taken along any piecewise \(C^1\) path in
 	\(\mathbb C_I\) joining \(0\) to \(z\). The product of the Volterra-type
 	integral operator and the slice regular composition operator is defined by
 	\[
 	V_{(g,\psi)}:=V _g C^\bullet_\psi .
 	\]
 	Equivalently, for \(f\in\mathcal{SR}(\mathbb H)\),
 	\[
 	(V_{(g,\psi)}f)(z)
 	=
 	\int_0^z
 	\bigl(g'\star(f\bullet\psi)\bigr)_I(w)\,\mathrm dw,
 	\qquad z\in\mathbb C_I,\quad I\in\mathbb S .
 	\]
 	In particular, if \(\psi\) is the identity map, then
 	\(V_{(g,\psi)}=V_g \).
 \end{definition}
The operators \(C^\bullet_\psi\), \(W_{(u,\psi)}\), \(V_g \), and
\(V_{(g,\psi)}\) are well defined and right linear. Indeed,
\(f\bullet\psi\) is slice regular, and the \(\star\)-product of slice regular
functions is slice regular. Moreover, the integrands
\(g'\star f\) and \(g'\star(f\bullet\psi)\) are slice regular, and hence their
restrictions to each slice are holomorphic. Therefore the corresponding slice
integrals are path independent and define holomorphic primitives on each slice,
normalized to vanish at \(0\). Since these primitives agree on the real axis, the
extension theorem yields unique slice regular functions on \(\mathbb H\); see \cite[Lemma~1.21]{MR3013643}.

	\subsection{Matrix realization and the testing quantity}
	
	The following matrix model follows directly from the slice-wise formula for the
	\(\star\)-product; see also \cite[Section~3.3]{MR4292259} for a closely related
	form. For an open set \(U\subset\mathbb C_I\), we write \(\mathcal O(U)\) for
	the space of holomorphic \(\mathbb C_I\)-valued functions on \(U\).
	
	In this subsection we fix $I,J\in\mathbb S$ with $J\perp I$. The map $\Phi_I$ below also depends on this auxiliary choice of $J$, but $J$ will remain fixed and will be suppressed from the notation. 	For \(f\in\mathcal{SR}(\HH)\), write \(f_I=F+GJ\). We define
	\begin{equation}\label{eq:PhiI-def}
		\Phi_I(f)(z)
		:=
		\begin{pmatrix}
			F(z)&G(z)\\[2pt]
			-\,G^\#(z)&F^\#(z)
		\end{pmatrix},
		\qquad z\in\mathbb C_I .
	\end{equation}
	
	\begin{proposition} 
		\label{thm:matrix-model}
		Fix \(I,J\in\mathbb S\) with \(I\perp J\). For
		\(f\in\mathcal{SR}(\HH)\), write \(f_I=F+GJ\), and let
		\(\Phi_I(f)\) be the matrix-valued function defined in \eqref{eq:PhiI-def}. Set
		\[
		\mathcal M_I
		:=
		\left\{
		\begin{pmatrix}
			U & V\\[2pt]
			-\,V^\# & U^\#
		\end{pmatrix}
		:\ U,V\in\mathcal O(\mathbb C_I)
		\right\}.
		\]
		Then
		\[
		\Phi_I:
		\bigl(\mathcal{SR}(\HH),\star\bigr)
		\longrightarrow
		\mathcal M_I
		\]
		is a multiplicative real-linear isomorphism onto \(\mathcal M_I\), where
		\(\mathcal M_I\) is endowed with pointwise matrix multiplication.
	\end{proposition}
	\begin{proof}

		The formula for the $\star$-product gives
		\[
		(f\star g)_I
		=
		(FH-GK^\#)+(FK+GH^\#)J
		\]
		whenever $f_I=F+GJ$ and $g_I=H+KJ$. A direct multiplication of the two matrices
		$\Phi_I(f)$ and $\Phi_I(g)$ gives exactly the matrix associated with this
		splitting of $f\star g$. Hence
		\[
		\Phi_I(f\star g)=\Phi_I(f)\Phi_I(g).
		\]

	The multiplication formula follows by direct computation from the slice
	product formula. Injectivity is immediate because the first row determines
	\(f_I\), hence \(f\). Conversely, for every pair \(F,G\in\mathcal O(\mathbb C_I)\),
	the extension theorem gives a unique slice regular extension of \(F+GJ\); see \cite[Lemma~1.21]{MR3013643}.
	Thus \(\Phi_I\) is onto \(\mathcal M_I\).
	\end{proof}
	
	Throughout the sequel, unless otherwise stated, we fix
\(I,J\in\mathbb S\) with \(J\perp I\).
Whenever slice regular functions \(g,f,\psi\) are involved, we use the
splitting notation
\begin{equation}\label{eq:basic-splitting-notation}
	f_I=F+GJ,\qquad
	\psi_I=P+QJ,
\end{equation}
where $
F,G,P,Q\in\mathcal O(\mathbb C_I)$. 

Set
\[
\Psi:=\Phi_I(\psi)
=
\begin{pmatrix}
	P&Q\\[2pt]
	-Q^\#&P^\#
\end{pmatrix}.
\]
Whenever a \(2\times2\) matrix-valued function occurs,
\(\|\cdot\|\) denotes any fixed matrix norm on
\(M_2(\mathbb C_I)\).

For later use, if \(\Theta\) is a \(2\times2\) matrix-valued function,
we write \(\Theta_{ij}\) for its \((i,j)\)-entry,
\(i,j\in\{1,2\}\). In particular,
\[
\psi_I
=
\Psi_{11}+\Psi_{12}J.
\]
	
	For each \(n\ge0\),  
	since \(\Phi_I(\psi^{\star n}) =\Psi ^n\), we have
	\begin{equation}\label{eq:Xn-Yn-def}
		(\psi^{\star n})_I  =\bigl(\Psi ^n\bigr)_{11}+ \bigl(\Psi ^n\bigr)_{12}J .
	\end{equation}
	For an entire function
	\[
	\varphi(\xi)=\sum_{n=0}^{\infty}\xi^n c_n,
	\qquad c_n\in\mathbb C_I,
	\]
	we define
	\[
	\varphi(\Psi)
	:=
	\sum_{n=0}^{\infty}\Psi^n c_n .
	\]
	Since the entries of $\Psi$ and the coefficients $c_n$ lie in  $\mathbb C_I$, this is the usual holomorphic functional calculus for the $2\times2$ matrix $\Psi$.

	For \(f\) written as in \eqref{eq:basic-splitting-notation}, we have $(f\bullet\psi)_I=	\bigl(\Phi_I(f\bullet\psi)\bigr)_{11}+	\bigl(\Phi_I(f\bullet\psi)\bigr)_{12}J$, where
	\begin{equation}\label{eq:Xf-Yf-global}
		\begin{aligned}
			\bigl(\Phi_I(f\bullet\psi)\bigr)_{11}
			&=
			\bigl(F(\Psi)\bigr)_{11}
			-
			\bigl(G^\#(\Psi)\bigr)_{12},
			\\[2pt]
			\bigl(\Phi_I(f\bullet\psi)\bigr)_{12}
			&=
			\bigl(F^\#(\Psi)\bigr)_{12}
			+
			\bigl(G(\Psi)\bigr)_{11}.
		\end{aligned}
	\end{equation}

	Indeed, write
	\begin{equation}\label{eq:FG-power-series} 
		F(\zeta)=\sum_{n=0}^{\infty}\zeta^n\alpha_n,
		\qquad
		G(\zeta)=\sum_{n=0}^{\infty}\zeta^n\beta_n,
		\qquad
		\alpha_n,\beta_n\in\mathbb C_I .
		\nonumber
	\end{equation}
	
	By \eqref{eq:Xn-Yn-def}, using \(Jc=\overline{c}J\) for \(c\in\mathbb C_I\), we obtain 
	\[
	\begin{aligned}
		(\psi^{\star n})_I(\alpha_n+\beta_nJ)
		={}&
		\left[
		\bigl(\Psi^n\bigr)_{11}\alpha_n
		-
		\bigl(\Psi^n\bigr)_{12}\overline{\beta_n}
		\right]+ 
		\left[
		\bigl(\Psi^n\bigr)_{12}\overline{\alpha_n}
		+
		\bigl(\Psi^n\bigr)_{11}\beta_n
		\right]J.
	\end{aligned}
	\]
	Consequently,
	\[
	\begin{aligned}
		(f\bullet\psi)_I
		={}&
		\sum_{n=0}^{\infty}
		\left[
		\bigl(\Psi^n\bigr)_{11}\alpha_n
		-
		\bigl(\Psi^n\bigr)_{12}\overline{\beta_n}
		\right]+
		\sum_{n=0}^{\infty}
		\left[
		\bigl(\Psi^n\bigr)_{12}\overline{\alpha_n}
		+
		\bigl(\Psi^n\bigr)_{11}\beta_n
		\right]J.
	\end{aligned}
	\]
	
	Since
	\[
	F^\#(\zeta)
	=
	\sum_{n=0}^{\infty}\zeta^n\overline{\alpha_n},
	\qquad
	G^\#(\zeta)
	=
	\sum_{n=0}^{\infty}\zeta^n\overline{\beta_n},
	\]
	this gives \eqref{eq:Xf-Yf-global}.
	
We shall use compact-open convergence only after restriction to a fixed slice.
By the representation formula, a sequence \(f_n\in\mathcal{SR}(\mathbb H)\)
converges to \(0\) locally uniformly on \(\mathbb H\) if and only if, for one, equivalently for every, splitting \( (f_n)_I=F_n+G_nJ\), both \(F_n\) and \(G_n\)
converge to \(0\) locally uniformly on \(\mathbb C_I\); see \cite[Remark~3.6]{lin}.

	We shall repeatedly use
\begin{equation}\label{eq:composition-first-row-equivalence}
	|(f\bullet\psi)_I(z)|^q
	\approx
	\left|\bigl(\Phi_I(f\bullet\psi)\bigr)_{11}(z)\right|^q
	+
	\left|\bigl(\Phi_I(f\bullet\psi)\bigr)_{12}(z)\right|^q.
\end{equation}

We shall use the following compact-open consequences repeatedly. For every compact set \(K\subset\mathbb C_I\), there exists a
constant \(C_K>0\) such that
\begin{equation}\label{eq:composition-compact-set-estimate}
	\sup_{z\in K}
	\|\Phi_I(f\bullet\psi)(z)\|
	\le
	C_K\|f\|_{F_\alpha^p(\mathbb H)},
	\qquad
	f\in F_\alpha^p(\mathbb H).
\end{equation}
Moreover, if \(f_n\to0\) uniformly on compact subsets of \(\mathbb H\)
in the Euclidean compact-open topology, then
\begin{equation}\label{eq:composition-compact-open-continuity}
	\Phi_I(f_n\bullet\psi)\longrightarrow0
	\quad\text{uniformly on }K.
\end{equation}

Indeed, by the preceding convention and equivalence, the splitting
components \(F_n\) and \(G_n\) converge to zero uniformly on compact
subsets of \(\mathbb C_I\). Hence the reflected functions
\(F_n^\#\) and \(G_n^\#\) also converge to zero uniformly on compact
subsets of \(\mathbb C_I\). Choose
\[
R>\sup_{z\in K}\|\Psi(z)\|.
\]
The Cauchy integral formula for the holomorphic matrix functional
calculus, applied on the circle \(|\zeta|=R\), together with
\eqref{eq:Xf-Yf-global} and the splitting norm comparison, gives
\eqref{eq:composition-compact-set-estimate} and
\eqref{eq:composition-compact-open-continuity}.

 
	\section{Slice Regular Composition Operators}
	\label{sec:composition-operators}
	
In this section we study the slice regular composition operator $
C_\psi^\bullet  $ 
acting between quaternionic Fock spaces, without assuming that
\(\psi\) preserves a fixed complex slice.

After fixing \(I\) and \(J\perp I\), the splitting \(\psi_I=P+QJ\) separates
the  case \(Q\equiv0\) from the  case
\(Q\not\equiv0\). In the latter case,
boundedness imposes strong necessary restrictions on the local eigenvalue branches of \(\Psi=\Phi_I(\psi)\). These restrictions are not sufficient
for boundedness, because the spectral-projection coefficients also
enter the pull-back measures; see
Example~\ref{ex:affine-branches-not-sufficient}.
	  
\subsection{The case where \(\psi\) preserves a fixed complex slice}
\label{subsec:composition-scalar-reductions}

	Fix \(I,J\in\mathbb S\) with \(J\perp I\), and write
	\begin{equation} 
		f_I=F+GJ,
		\qquad
		\psi_I=P+QJ,
		\nonumber
	\end{equation}
	where \(F,G,P,Q\in\mathcal O(\mathbb C_I)\). Set
	\[
	\Psi:=\Phi_I(\psi)
	=
	\begin{pmatrix}
		P&Q\\[2pt]
		-Q^\#&P^\#
	\end{pmatrix},
	\]
	and define
	\begin{equation}\label{eq:composition-tau-delta-Delta}
		\tau:=P+P^\#,
		\qquad
		\delta:=PP^\#+QQ^\#,
		\qquad
		\Delta:=\tau^2-4\delta.
	\end{equation}
	Equivalently,
	\begin{equation}\label{eq:composition-Delta-expanded}
		\Delta=(P-P^\#)^2-4QQ^\#.
	\end{equation}
	Notice that \(\tau^\#=\tau\) and \(\Delta^\#=\Delta\).
	
We first note that the degenerate-discriminant case is contained in the case where
the symbol preserves a fixed complex slice. 	Assume first that \(\Delta\equiv0\). Restricting
\eqref{eq:composition-Delta-expanded} to the real axis gives
\[
0
=
-4\bigl((\operatorname{Im}P(x))^2+|Q(x)|^2\bigr),
\qquad x\in\mathbb R.
\]
Hence
\[
Q(x)=0,
\qquad
P(x)\in\mathbb R,
\qquad x\in\mathbb R.
\]
Since \(Q\) and \(P-P^\#\) are holomorphic and vanish on the real
axis, the identity theorem yields
\[
Q\equiv0,
\qquad
P=P^\#.
\]
Thus \(\psi\) is intrinsic by \cite[Lemma~2.9]{lin}.

More generally, suppose that \(Q\equiv0\). Then \(\psi\) preserves the fixed
complex slice \(\mathbb C_I\), and
\[
\Psi=\operatorname{diag}(P,P^\#).
\]
In this case \eqref{eq:Xf-Yf-global} reduces to
\begin{equation}\label{eq:composition-fixed-slice-scalar-formula}
	(f\bullet\psi)_I
	=
	(F\circ P)+(G\circ P)J .
\end{equation}
Thus, the quaternionic composition problem reduces, on the fixed slice, to the scalar pull-back map induced by \(P\).

We now formulate this reduction in terms of the associated pull-back measure.
For a Borel set \(\mathcal A\subset\mathbb C_I\), define
\[
\mu_{P,I}(\mathcal A)
:=
\int_{P^{-1}(\mathcal A)}
\mathrm e^{-\frac{\alpha q}{2}|z|^2}
\,\mathrm d m_{2,I}(z),
\]
and set
\begin{equation}\label{eq:composition-scalar-weighted-pullback}
	\mathrm d\nu_{P,I}(\zeta)
	:=
	\mathrm e^{\frac{\alpha q}{2}|\zeta|^2}
	\,\mathrm d\mu_{P,I}(\zeta).
	\nonumber
\end{equation}
Consequently, the scalar composition operator
\[
C_P:
F_\alpha^p(\mathbb C_I)\longrightarrow F_\alpha^q(\mathbb C_I),
\qquad C_PH=H\circ P,
\]
is bounded if and only if \(\nu_{P,I}\) is a \((p,q)\)-Fock--Carleson
measure; similarly, \(C_P\) is compact if and only if \(\nu_{P,I}\) is a
vanishing \((p,q)\)-Fock--Carleson measure; see 
\cite[Theorems~2.6 and~2.7]{MR3248473} see, also,
\cite[Theorems~1 and~2]{ueki2007weighted}.

In particular, by \eqref{eq:composition-testing-quantity} and \eqref{eq:composition-fixed-slice-scalar-formula}, we obtain 
\begin{equation}\label{eq:composition-scalar-Berezin-reduction}
	\mathcal B_{\psi,I}(w)
	=
	\int_{\mathbb C_I}
	|k_w(\zeta)|^q
	\mathrm e^{-\frac{\alpha q}{2}|\zeta|^2}
	\,\mathrm d\nu_{P,I}(\zeta).
	\nonumber
\end{equation}
Thus \(\mathcal B_{\psi,I}\) is exactly the Berezin-type transform associated
with \(\nu_{P,I}\).

Moreover, if \(f_I=F+GJ\), then \eqref{eq:composition-fixed-slice-scalar-formula}
and the elementary splitting estimate yield
\begin{equation}\label{eq:composition-scalar-operator-reduction}
	\begin{aligned}
		&\int_{\mathbb C_I}
		|(f\bullet\psi)_I(z)|^q
		\mathrm e^{-\frac{\alpha q}{2}|z|^2}
		\,\mathrm d m_{2,I}(z)    \approx
		\int_{\mathbb C_I}
		\bigl(|F(\zeta)|^q+|G(\zeta)|^q\bigr)
		\mathrm e^{-\frac{\alpha q}{2}|\zeta|^2}
		\,\mathrm d\nu_{P,I}(\zeta).
	\end{aligned}
\end{equation}
It follows that
\begin{equation}\label{eq:composition-scalar-operator-equivalence}
	C_\psi^\bullet
	\text{ is bounded, respectively compact}
	\quad\Longleftrightarrow\quad
	C_P
	\text{ is bounded, respectively compact}.
\end{equation}
Indeed, boundedness of \(C_\psi^\bullet\) applied to the slice regular extension of
\(H\in F_\alpha^p(\mathbb C_I)\) gives boundedness of \(C_P\). Conversely,
boundedness of \(C_P\), applied separately to the splitting components \(F\) and
\(G\) in \eqref{eq:composition-scalar-operator-reduction}, together with
\eqref{eq:splitting-fock-norm-comparison}, gives boundedness of
\(C_\psi^\bullet\). The compactness equivalence follows by the same argument.

Therefore, when \(Q\equiv0\), Theorems~\ref{thm:composition-p-le-q} and
\ref{thm:composition-q-less-p}, including the norm estimates, follow from the
classical complex Fock--Carleson theorems
\cite[Theorems~2.6 and~2.7]{MR3248473}.
\begin{corollary}\label{cor:composition-scalar-affine-rigidity}
Assume that, with respect to the fixed splitting 
\(\psi_I=P+QJ\), one has \(Q\equiv0\). Equivalently,
\(\psi\) preserves the fixed complex slice \(\mathbb C_I\).
	
	\begin{enumerate}[(i)]
		\item Let \(0<p\le q<\infty\). Then $
		C_\psi^\bullet:
		F_\alpha^p(\mathbb H)\longrightarrow F_\alpha^q(\mathbb H)$ 
		is bounded if and only if    \(P(z)=az+b\) where \(|a|\le1\) and \(b=0\) whenever \(|a|=1\).
		Moreover, \(C_\psi^\bullet\) is compact if and only if  \(P(z)=az+b\) with \(|a|<1\).
		
		\item Let \(0<q<p<\infty\). Then boundedness and compactness of $
		C_\psi^\bullet:
		F_\alpha^p(\mathbb H)\longrightarrow F_\alpha^q(\mathbb H)$ 
		are equivalent, and they hold if and only if  \(P(z)=az+b\) with \(|a|<1\).
	\end{enumerate}
	
	In all the preceding formulas, \(a,b\in\mathbb C_I\). If, in
	addition, \(\Delta\equiv0\), then \(a,b\in\mathbb R\).
\end{corollary}

\begin{proof}
	By \eqref{eq:composition-scalar-operator-equivalence},
	\(C_\psi^\bullet\) is bounded, respectively compact, if and only if $
	C_P:
	F_\alpha^p(\mathbb C_I)\longrightarrow
	F_\alpha^q(\mathbb C_I)$ has the corresponding property. The scalar 
	classification in \cite[Corollaries~3.5 and~3.6]{MR3905527}, after a  harmless 
	change of normalization, gives \textup{(i)} and \textup{(ii)}.
	
	If \(\Delta\equiv0\), then \(P=P^\#\).   Hence
	\(a,b\in\mathbb R\).
\end{proof}
In the next subsections, we work in the two-branch case. More precisely,
we keep the fixed pair \(I,J\in\mathbb S\), \(J\perp I\), and the splitting
\[
\psi_I=P+QJ,\qquad P,Q\in\mathcal O(\mathbb C_I),
\]
introduced above, and assume
\begin{equation}\label{eq:composition-genuine-case}
	Q\not\equiv0.
\end{equation}
Then \(\Delta\not\equiv0\). All quantities
\(\Psi,\tau,\Delta,\lambda_\pm,c_\pm,d_\pm\) are understood with respect to this
fixed splitting.
	
\subsection{The local two-branch representation}
	\label{subsec:composition-two-branch-formula}
	
	Set
	\[
	\mathcal N:=\{z\in\mathbb C_I:\Delta(z)\ne0\}.
	\]
A \emph{nondegenerate local region} is a simply connected domain
\(U\subset\mathcal N\) satisfying \(U=U^\#\).
	
	Let \(U\) be a nondegenerate local region. Choose
	\[
	\rho\in\mathcal O(U),
	\qquad
	\rho^2=\Delta,
	\]
	and set
	\[
	\lambda_\pm:=\frac{\tau\pm\rho}{2}.
	\]
We call \(\lambda_\pm\) the local eigenvalue branches of
\(\Psi\) on \(U\). If a pair of such branches can be chosen holomorphically on
the whole fixed slice \(\mathbb C_I\), we call them global eigenvalue functions.
	
	The corresponding spectral projections are
	\begin{equation}\label{eq:composition-spectral-projections}
		E_+
		:=
		\frac{\Psi-\lambda_-\mathrm I_2}{\rho},
		\qquad
		E_-
		:=
		\frac{\Psi-\lambda_+\mathrm I_2}{-\rho}
		=
		\mathrm I_2-E_+.
	\end{equation}
	Define
	\begin{equation}\label{eq:composition-local-coefficients}
		c_\pm:=(E_\pm)_{11},
		\qquad
		d_\pm:=(E_\pm)_{12}.
	\end{equation}
 	The holomorphic functional calculus gives
 \[
 H(\Psi)=H(\lambda_+)E_++H(\lambda_-)E_-
 \]
 for every entire \(\mathbb C_I\)-valued function \(H\). Substituting this
 identity into \eqref{eq:Xf-Yf-global} gives
 \eqref{eq:composition-local-four-term-11} and
 \eqref{eq:composition-local-four-term-12}.

		Let \(U\) be a nondegenerate local region. Then, on \(U\),
		\begin{equation}\label{eq:composition-local-four-term-11}
			\bigl(\Phi_I(f\bullet\psi)\bigr)_{11}
			=
			\sum_{\sigma=\pm}
			\left[
			c_\sigma F(\lambda_\sigma)
			-d_\sigma G^\#(\lambda_\sigma)
			\right],
		\end{equation}
		and
		\begin{equation}\label{eq:composition-local-four-term-12}
			\bigl(\Phi_I(f\bullet\psi)\bigr)_{12}
			=
			\sum_{\sigma=\pm}
			\left[
			d_\sigma F^\#(\lambda_\sigma)
			+c_\sigma G(\lambda_\sigma)
			\right].
		\end{equation}

	\begin{remark} 
		\label{rem:composition-branch-free-kernel}
		For \(z,w\in\mathbb C_I\), the matrix functional calculus gives
		\begin{equation}\label{eq:composition-kernel-matrix-exponential}
			\begin{aligned}
				\bigl(\Phi_I(k_w\bullet\psi)\bigr)_{11}(z)
				&=
				\mathrm e^{-\frac{\alpha}{2}|w|^2}
				\left(\exp\bigl(\alpha\overline w\,\Psi(z)\bigr)\right)_{11},\\[2pt]
				\bigl(\Phi_I(k_w\bullet\psi)\bigr)_{12}(z)
				&=
				\mathrm e^{-\frac{\alpha}{2}|w|^2}
				\left(\exp\bigl(\alpha w\,\Psi(z)\bigr)\right)_{12}.
			\end{aligned}
		\end{equation}
		The Cayley--Hamilton identity
		\[
		\left(\Psi-\frac{\tau}{2}\mathrm I_2\right)^2
		=
		\frac{\Delta}{4}\mathrm I_2
		\]
		therefore yields global formulas involving only the entire functions
		\[
		\mathcal C(\zeta):=\sum_{m=0}^\infty\frac{\zeta^m}{(2m)!},
		\qquad
		\mathcal S(\zeta):=\sum_{m=0}^\infty\frac{\zeta^m}{(2m+1)!}.
		\]
		In particular,
		\begin{align*}
			\bigl(\Phi_I(k_w\bullet\psi)\bigr)_{11}(z)
			={}&
			\mathrm e^{-\frac{\alpha}{2}|w|^2+
				\frac{\alpha\overline w}{2}\tau(z)}
			\Bigg[
			\mathcal C\!\left(\frac{\alpha^2\overline w^{\,2}\Delta(z)}{4}\right)\\
			&\hspace{3.7cm}
			+\alpha\overline w\,
			\mathcal S\!\left(\frac{\alpha^2\overline w^{\,2}\Delta(z)}{4}\right)
			\left(P(z)-\frac{\tau(z)}2\right)
			\Bigg],\\
			\bigl(\Phi_I(k_w\bullet\psi)\bigr)_{12}(z)
			={}&
			\mathrm e^{-\frac{\alpha}{2}|w|^2+
				\frac{\alpha w}{2}\tau(z)}
			\alpha w\,
			\mathcal S\!\left(\frac{\alpha^2w^2\Delta(z)}{4}\right)Q(z).
		\end{align*}
		These formulas are used near the zero set of \(\Delta\), where a holomorphic
		square-root branch need not exist.
	\end{remark}
	
	\begin{example}\label{ex:psi-iz-plus-j}
		Fix \(I,J\in\mathbb S\) with \(I\perp J\), and let
		\[
		\psi_I(z)=Iz+J.
		\]
		Then
		\[
		\Psi(z)=
		\begin{pmatrix}
			Iz&1\\
			-1&-Iz
		\end{pmatrix},
		\qquad
		\Psi(z)^2=-(z^2+1)\mathrm I_2.
		\]
		Hence
		\[
		\tau\equiv0,
		\qquad
		\Delta(z)=-4(z^2+1).
		\]
		The local eigenvalue branches are therefore
		\[
		\lambda_\pm(z)=\pm I\sqrt{z^2+1},
		\]
		and cannot be chosen holomorphically across the points \(z=\pm I\).
		Nevertheless, the  formulas \eqref{eq:composition-kernel-matrix-exponential}  remain holomorphic
		across these points. This elementary example motivates the
		good/singular decomposition introduced below.
	\end{example}
	
\subsection{Rigidity of eigenvalue branches}
	\label{subsec:composition-rigidity}
The rigidity obtained in this subsection is necessary but not
sufficient for boundedness. Indeed, in the case
\(Q\not\equiv0\), the fixed-slice formulas
\eqref{eq:composition-local-four-term-11} and
\eqref{eq:composition-local-four-term-12} reduce the slice regular operator problem
locally to differences of complex weighted composition operators induced by the local eigenvalue branches \(\lambda_+\) and \(\lambda_-\). Thus the  boundedness also depends on the spectral coefficients
\(c_\pm\) and \(d_\pm\).
Consequently, even when both eigenvalue branches are globally defined affine functions
with slopes of modulus strictly less than one, rapid growth of the
spectral coefficients may prevent \(C_\psi^\bullet\) from being
bounded; see Example~\ref{ex:affine-branches-not-sufficient}.

	For a holomorphic function \(H\) on a nondegenerate local region \(U\) and a
	holomorphic map \(\lambda:U\to\mathbb C_I\), set
	\begin{equation}\label{eq:composition-M-def}
		M_z(H,\lambda)
		:=
		|H(z)|
		\exp\!\left(
		\frac\alpha2\bigl(|\lambda(z)|^2-|z|^2\bigr)
		\right). 
	\end{equation}
	We also write
	\begin{equation}\label{eq:composition-separation-def}
		d(u,v):=\frac{|u-v|^2}{2(2+|u-v|^2)}.
		\nonumber
	\end{equation}
 
		Define
	\begin{equation}\label{eq:composition-B-infty}
		\mathfrak B_\infty
		:=
		\left(
		\sup_{w\in\mathbb C_I}\mathcal B_{\psi,I}(w)
		\right)^{1/q},\nonumber
	\end{equation}
	and, when \(0<q<p<\infty\),
	\begin{equation}\label{eq:composition-B-s}
		s:=\frac{p}{p-q},
		\qquad
		\mathfrak B_s
		:=
		\left(
		\int_{\mathbb C_I}\mathcal B_{\psi,I}(w)^s
		\,\mathrm d m_{2,I}(w)
		\right)^{1/(qs)}.
		\nonumber
	\end{equation}
	
	We next introduce two global entire functions:
\[
\mathcal H:=-\Delta QQ^\#,
\qquad
\mathcal L:=-\Delta Q^2.
\]
Under \eqref{eq:composition-genuine-case}, neither function is
identically zero. Moreover, on every nondegenerate local region,
\begin{equation}\label{eq:HL-local-restriction}
	\mathcal H=c_+c_-\rho^4,
	\qquad
	\mathcal L=d_+d_-\rho^4.
\end{equation}

	\begin{proof}[Proof of Theorem~\ref{thm:composition-spectral-rigidity}]
		Fix a nondegenerate local region \(U\), a branch \(\rho^2=\Delta\), and
		\(z\in U\). For fixed \(w\in\mathbb C_I\), by the  local estimate for   holomorphic functions, see  \cite[Lemma~2.32]{MR2934601}, for a fixed \(r\in(0,1]\),
		\[
		\left|\bigl(\Phi_I(k_w\bullet\psi)\bigr)_{11}(z)\right|^q
		\mathrm e^{-\frac{\alpha q}{2}|z|^2}
		\lesssim
		\int_{D_I(z,r)}
		\left|\bigl(\Phi_I(k_w\bullet\psi)\bigr)_{11}(\zeta)\right|^q
		\mathrm e^{-\frac{\alpha q}{2}|\zeta|^2}
		\,\mathrm d m_{2,I}(\zeta).
		\]
		By \eqref{eq:composition-first-row-equivalence} and
		\eqref{eq:composition-testing-quantity},
		\begin{equation}\label{eq:composition-local-pointwise-test}
			\left|\bigl(\Phi_I(k_w\bullet\psi)\bigr)_{11}(z)\right|^q
			\mathrm e^{-\frac{\alpha q}{2}|z|^2}
			\lesssim
			\mathcal B_{\psi,I}(w).
			\nonumber
		\end{equation}
		Since
		\[
		\bigl(\Phi_I(k_w\bullet\psi)\bigr)_{11}(z)
		=
		\sum_{\sigma=\pm}c_\sigma(z)k_w(\lambda_\sigma(z)),
		\]
		we obtain
		\begin{equation}\label{eq:composition-local-kernel-test-c}
			\left|
			\sum_{\sigma=\pm}
			\overline{c_\sigma(z)}K_{\lambda_\sigma(z)}(w)
			\right|^q
			\mathrm e^{-\frac{\alpha q}{2}|w|^2}
			\lesssim
			\mathrm e^{\frac{\alpha q}{2}|z|^2}
			\mathcal B_{\psi,I}(w).
		\end{equation}
		Applying the same argument to the \((1,2)\)-entry and then replacing
		\(w\) by \(\overline w\), one obtains the analogous estimate with
		\(d_\sigma\) in place of \(c_\sigma\).
		
		Assume first that \(\mathfrak B_\infty<\infty\). Taking the supremum over
		\(w\) in \eqref{eq:composition-local-kernel-test-c}, and using the
		two-kernel separation estimate \cite[Lemma~2.1]{MR3895397}, gives
		\begin{equation}\label{eq:composition-separated-c-infty}
			d(\lambda_+(z),\lambda_-(z))
			\sum_{\sigma=\pm}M_z(c_\sigma,\lambda_\sigma)
			\lesssim
			\mathfrak B_\infty.
		\end{equation}
		Similarly,
		\begin{equation}\label{eq:composition-separated-d-infty}
			d(\lambda_+(z),\lambda_-(z))
			\sum_{\sigma=\pm}M_z(d_\sigma,\lambda_\sigma)
			\lesssim
			\mathfrak B_\infty.
		\end{equation}
		
		Assume next that \(0<q<p<\infty\) and \(\mathfrak B_s<\infty\). Raising
		\eqref{eq:composition-local-kernel-test-c} to the power \(s\), integrating in
		\(w\), and using the embedding
		\(F_\alpha^{qs}(\mathbb C_I)\hookrightarrow F_\alpha^\infty(\mathbb C_I)\),
		we obtain
		\begin{equation}\label{eq:composition-separated-c-Ls}
			d(\lambda_+(z),\lambda_-(z))
			\sum_{\sigma=\pm}M_z(c_\sigma,\lambda_\sigma)
			\lesssim
			\mathfrak B_s.
		\end{equation}
		The reflected estimate for the \((1,2)\)-entry gives
		\begin{equation}\label{eq:composition-separated-d-Ls}
			d(\lambda_+(z),\lambda_-(z))
			\sum_{\sigma=\pm}M_z(d_\sigma,\lambda_\sigma)
			\lesssim
			\mathfrak B_s.
		\end{equation}
		Thus, in either case, the right-hand sides of
		\eqref{eq:composition-separated-c-infty}--
		\eqref{eq:composition-separated-d-Ls} are bounded by a constant independent
		of \(U\) and of the choice of the sign of \(\rho\).
		
		We now use \(\mathcal H=-\Delta QQ^\#\not\equiv0\). Since
		\[
		\lambda_+-\lambda_-=\rho,
		\qquad
		d(\lambda_+,\lambda_-)
		=
		\frac{|\rho|^2}{2(2+|\rho|^2)},
		\]
		the preceding estimates imply
		\[
		|c_\pm(z)|
		\lesssim
		\frac{2+|\rho(z)|^2}{|\rho(z)|^2}
		\exp\!\left
		\{\frac\alpha2\bigl(|z|^2-|\lambda_\pm(z)|^2\bigr)\right\}.
		\]
		Multiplying the two estimates and using
		\(\mathcal H=c_+c_-\rho^4\), we obtain
		\begin{equation}\label{eq:composition-HC-growth-prelog}
			|\mathcal H (z)|
			\lesssim
			(2+|\rho(z)|^2)^2
			\exp\!\left\{
			\alpha|z|^2
			-\frac\alpha2\sum_{\sigma=\pm}|\lambda_\sigma(z)|^2
			\right\}.
			\nonumber
		\end{equation}
		Since
		\[
		\sum_{\sigma=\pm}|\lambda_\sigma|^2
		=
		\frac12|\tau|^2+\frac12|\rho|^2,
		\]
		we get
		\[
		\log|\mathcal H (z)|
		+\frac\alpha4|\tau(z)|^2
		+\frac\alpha4|\rho(z)|^2
		-\alpha|z|^2
		\le
		C+2\log(2+|\rho(z)|^2).
		\]
		Using
		\[
		2\log(2+x)\le\frac\alpha8x+C,
		\qquad x\ge0,
		\]
		and \(|\rho|^2=|\Delta|\), we obtain the branch-independent estimate
		\begin{equation}\label{eq:composition-global-growth}
			\log|\mathcal H (z)|
			+\frac\alpha4|\tau(z)|^2
			+\frac\alpha8|\Delta(z)|
			-\alpha|z|^2
			\le C,
			\qquad z\in\mathcal N.
		\end{equation}
		
		To prove \textup{(i)}, drop the nonnegative term involving \(\tau\) and
		average \eqref{eq:composition-global-growth} over circles \(|z|=R\) avoiding
		the zeros of \(\Delta\mathcal H\). Jensen's formula for the nonzero entire
		function \(\mathcal H\) gives
		\[
		\frac1{2\pi}\int_0^{2\pi}
		|\Delta(Re^{I\theta})|\,\mathrm d\theta
		\lesssim 1+R^2.
		\]
		If \(\mathcal H\) has a zero of order \(m\) at the origin, apply
		Jensen's formula to \(\mathcal H(z)/z^m\). Thus the circular mean of
		\(\log|\mathcal H|\) is bounded from below, up to the harmless term
		\(m\log R\).
		
		If \(\Delta(z)=\sum_{n\ge0}d_nz^n\), Cauchy's formula yields
		\[
		|d_n|
		\lesssim
		\frac{1+R^2}{R^n}.
		\]
		Letting \(R\to\infty\), we get \(d_n=0\) for \(n\ge3\). Hence
		\[
		\Delta(z)=d_2z^2+d_1z+d_0.
		\]
		The identity \(\Delta^\#=\Delta\) gives \(d_0,d_1,d_2\in\mathbb R\). On the
		real axis,
		\[
		\Delta(x)
		=-4\bigl((\operatorname{Im}P(x))^2+|Q(x)|^2\bigr)
		\le0.
		\]
		Thus \(d_2\le0\). If \(d_2=0\) and \(d_1\ne0\), then \(\Delta\) is a
		nonconstant real affine function that is nonpositive on all of \(\mathbb R\),
		a contradiction. Therefore either \(\Delta\) is constant or \(d_2<0\). Since
		\(Q\not\equiv0\), the constant cannot be zero.
		
		To prove \textup{(ii)}, drop the term involving \(\Delta\) in
		\eqref{eq:composition-global-growth} and average over \(|z|=R\). Jensen's
		formula again gives
		\[
		\frac1{2\pi}\int_0^{2\pi}
		|\tau(Re^{I\theta})|^2\,\mathrm d\theta
		\lesssim 1+R^2.
		\]
		Writing \(\tau(z)=\sum_{n\ge0}t_nz^n\), Parseval's identity yields
		\[
		\sum_{n\ge0}|t_n|^2R^{2n}\lesssim1+R^2.
		\]
		Hence \(t_n=0\) for \(n\ge2\), so
		\[
		\tau(z)=t_1z+t_0.
		\]
		Since \(\tau^\#=\tau\), one has \(t_0,t_1\in\mathbb R\).
		
		We next prove \textup{(iii)}. Suppose that \(\Delta\) is nonconstant. Then
		\[
		\Delta(z)=d_2z^2+d_1z+d_0,
		\qquad d_2<0.
		\]
		Choose \(R>0\) containing the zeros of \(\Delta\), and put
		\[
		V_0:=\{z\in\mathbb C_I:|z|>R\}\setminus[R,\infty).
		\]
Then \(V_0\) is unbounded and simply connected, and \(V_0^\#=V_0\). 
		Choose \(\rho\in\mathcal O(V_0)\) with \(\rho^2=\Delta\). There exists a
		nonzero purely imaginary \(s_0\in\mathbb C_I\) with \(s_0^2=d_2\) such that
		\[
		\rho(z)=s_0z+\gamma+O(1/z),
		\qquad
		\gamma=\frac{d_1}{2s_0}.
		\]
		Consequently,
		\[
		\lambda_\pm(z)
		=
		a_\pm z+b_\pm+O(1/z),
		\]
		where
		\[
		a_\pm=\frac{t_1\pm s_0}{2},
		\qquad
		b_\pm=\frac{t_0\pm\gamma}{2}.
		\]
		Since \(t_1\) is real and \(s_0\) is nonzero and purely imaginary,
		\[
		a_\pm\ne0,
		\qquad
		|a_+|=|a_-|.
		\]
		Suppose that this common modulus is larger than one. Then there exist
		\(\eta>0\) and \(R_1>R\) such that
		\[
		\sum_{\sigma=\pm}|\lambda_\sigma(z)|^2-2|z|^2
		\ge\eta|z|^2,
		\qquad z\in V_0,\ |z|>R_1.
		\]
		On the exterior part of \(V_0\), the separation factor is bounded below.
		Hence the coefficient estimates and
		\(\mathcal H =c_+c_-\rho^4\) give
		\[
		|\mathcal H (z)|
		\lesssim
		(1+|z|)^4\mathrm e^{-\frac{\alpha\eta}{2}|z|^2}.
		\]
		By continuity the same decay holds on the removed ray. Thus
		\(\mathcal H (z)\to0\) as \(|z|\to\infty\) in the whole plane. Liouville's
		theorem would imply \(\mathcal H\equiv0\), a contradiction. Therefore
\(|a_\pm|\le1\). This proves \textup{(iii)}.

It remains to prove \textup{(iv)}. Suppose that \(\Delta\) has no zeros on
\(\mathbb C_I\). By \textup{(i)}, \(\Delta\) is either a nonzero constant or a
quadratic polynomial with negative leading coefficient. The second alternative
is impossible, since every nonconstant polynomial on \(\mathbb C_I\) has a
zero. Hence $
\Delta\equiv\Delta_0\ne0 $. 
Choose a constant square root \(\rho\) of \(\Delta_0\). By \textup{(ii)},
\[
\tau(z)=t_1z+t_0,
\qquad t_0,t_1\in\mathbb R.
\]
Thus the two eigenvalue functions are globally defined and have the form
\[
\lambda_\pm(z)
=
\frac{\tau(z)\pm\rho}{2}
=
az+b_\pm,
\qquad
a:=\frac{t_1}{2},
\qquad
b_\pm:=\frac{t_0\pm\rho}{2}.
\]
In particular, \(a,b_\pm\in\mathbb C_I\).

The estimate \(|a|\le1\) follows from the same exterior-decay argument used
above in the proof of \textup{(iii)}. Indeed, if \(|a|>1\), then
\(\lambda_+-\lambda_-=\rho\ne0\), so the separation factor
\(d(\lambda_+(z),\lambda_-(z))\) is bounded below by a positive constant. Also,
for some \(\eta>0\) and all sufficiently large \(|z|\),
\[
\sum_{\sigma=\pm}|\lambda_\sigma(z)|^2-2|z|^2
\ge \eta |z|^2 .
\]
Using the separated coefficient estimates and
\(\mathcal H=c_+c_-\rho^4\), where now \(\rho\) is constant, we obtain
\[
|\mathcal H(z)|
\lesssim
\exp\left\{
\alpha |z|^2
-\frac{\alpha}{2}\sum_{\sigma=\pm}|\lambda_\sigma(z)|^2
\right\}
\lesssim
\mathrm e^{-\frac{\alpha\eta}{2}|z|^2}.
\]
Hence \(\mathcal H(z)\to0\) as \(|z|\to\infty\). Since \(\mathcal H\) is
entire, Liouville's theorem gives \(\mathcal H\equiv0\), contradicting
\[
\mathcal H=-\Delta_0 QQ^\#\not\equiv0
\]
because \(\Delta_0\ne0\) and \(Q\not\equiv0\). Therefore \(|a|\le1\), and
\textup{(iv)} follows.
\end{proof} 
\begin{corollary}\label{cor:composition-zero-free-discriminant}
	Assume the hypotheses of Theorem~\ref{thm:composition-spectral-rigidity},
	and suppose that \(\Delta\) has no zeros on \(\mathbb C_I\). Then there exist
	\(a,b\in\mathbb R\), \(c>0\), and an entire function
	\(h:\mathbb C_I\to\mathbb C_I\) with \(h^\#=h\), such that
	\[
	P(z)=az+b+Ih(z),
	\qquad
	QQ^\#=c^2-h^2,
	\qquad
	|a|\le1.
	\]
	Equivalently,
	\[
	\psi_I(z)=az+b+Ih(z)+Q(z)J .
	\]
	
	If, in addition, \(P^\#=P\), then \(h\equiv0\), and hence
	\(QQ^\#=c^2\). Consequently,
	\[
	Q=c\,\mathrm e^{I\ell}
	\]
	for some entire function \(\ell:\mathbb C_I\to\mathbb C_I\) satisfying
	\(\ell^\#=\ell\).
\end{corollary}

\begin{proof}
	By Theorem~\ref{thm:composition-spectral-rigidity}\textup{(iv)}, the
	zero-free assumption implies that $
		\Delta\equiv\Delta_0\ne0 $. 
	Moreover, after choosing a constant square root of \(\Delta_0\), the two
	global eigenvalue functions are affine with common slope of modulus at most
	one. On the other hand, by
	Theorem~\ref{thm:composition-spectral-rigidity}\textup{(ii)},
	\[
		\tau(z)=t_1z+t_0,
		\qquad t_0,t_1\in\mathbb R .
	\]
	Thus the common slope is \(t_1/2\). Put
	\[
		a:=\frac{t_1}{2},
		\qquad
		b:=\frac{t_0}{2}.
	\]
	Then \(a,b\in\mathbb R\) and \(|a|\le1\).
	
	We next determine the form of \(P\). 	Since \(\Delta(x)\le0\) for real \(x\), the nonzero constant \(\Delta_0\) is
		negative. Write \(\Delta_0=-4c^2\), \(c>0\), and choose \(\rho=2Ic\). The
		identity
		\[
		\left(P-\frac\tau2\right)^\#
		=-\left(P-\frac\tau2\right)
		\]
		shows that
		\[
		P(z)=az+b+Ih(z)
		\]
		for an entire \(h\) satisfying \(h^\#=h\). Finally,
		\eqref{eq:composition-Delta-expanded} gives
		\[
		QQ^\#=c^2-h^2.
		\]
	
	If \(P^\#=P\), then \(P-\tau/2\equiv0\), so \(h\equiv0\), and consequently
\(QQ^\#=c^2\). Thus \(Q/c\) is zero-free on the simply connected plane
	\(\mathbb C_I\), and hence \(Q/c=\mathrm e^H\) for some entire
	\(H:\mathbb C_I\to\mathbb C_I\). The identity \(QQ^\#=c^2\) gives
	\[
	\mathrm e^{H+H^\#}=1.
	\]
	Thus \(H+H^\#\) is constant with values in \(2\pi I\mathbb Z\). Evaluating on
the real axis gives \(H(x)+H^\#(x)=2\operatorname{Re}H(x)\in\mathbb R\).
Hence this constant is both real and purely imaginary, so it is \(0\). Therefore
	\(H^\#=-H\). Setting \(\ell:=-IH\), we get \(\ell^\#=\ell\) and $
	Q=c\,\mathrm e^{I\ell}$. 
	This completes the proof.
\end{proof}
	 
\subsection{The good/singular decomposition}
	\label{subsec:composition-good-singular}
 
We next introduce a geometric decomposition of the fixed slice. In the
case \(Q\not\equiv0\), Theorem~\ref{thm:composition-spectral-rigidity}
implies that the zero set \(Z(\Delta)\) is finite, with at most two points.
Since the local spectral representation
\eqref{eq:composition-local-four-term-11}--\eqref{eq:composition-local-four-term-12}
requires a holomorphic branch of \(\sqrt{\Delta}\), we separate off a small
compact neighborhood \(E\) of \(Z(\Delta)\). The set \(E\) may be chosen with
arbitrarily small planar measure. On the remaining good regions, the chosen local eigenvalue branches and the
corresponding spectral coefficients are holomorphic, and the two local eigenvalue
branches are uniformly separated.

	\begin{lemma}\label{lem:composition-good-singular-decomposition}
		Assume
		\[
		\Delta(z)=d_2z^2+d_1z+d_0,
		\qquad d_0,d_1,d_2\in\mathbb R,
		\qquad d_2<0.
		\]
	There exists a compact set \(E\subset\mathbb C_I\) with \(E=E^\#\) and 
\(Z(\Delta)\subset E\), which can be chosen as an arbitrarily small finite union of
		pairwise disjoint closed disks centered at the zeros of \(\Delta\). Moreover,
		there exist a null set \(\Gamma\subset\mathbb C_I\) and two nondegenerate
		local regions \(U_0,U_1\subset\mathbb C_I\setminus E\), with \(U_0\)
		unbounded and \(U_1\) bounded, such that
		\[
		\mathbb C_I\setminus(E\cup\Gamma)=U_0\cup U_1.
		\]
		In addition, there is \(\varepsilon_0>0\) for which
		\[
		|\Delta(z)|\ge\varepsilon_0,
		\qquad z\in U_0\cup U_1.
		\]
	\end{lemma}
	
	\begin{proof}
		The zero set of \(\Delta\) is finite and $Z(\Delta)^\#=Z(\Delta)$. Choose a finite union \(E_0\) of pairwise disjoint closed disks centered at
zeros of \(\Delta\), with arbitrarily small  radius, and put $
E:=E_0\cup E_0^\# $. 
After decreasing the radii if necessary, \(E\) is still a finite union of
pairwise disjoint closed disks, contains \(Z(\Delta)\), satisfies \(E=E^\#\),
and has arbitrarily small planar measure. Since \(|\Delta(z)|\to\infty\)
		as \(|z|\to\infty\), one has
		\[
		\inf_{\mathbb C_I\setminus E}|\Delta|>0.
		\]
		Choose \(R\) so that \(E\subset D(0,R)\), and put
		\[
		U_0:=\{z:|z|>R\}\setminus[R,\infty).
		\]
	Inside \(D(0,R)\setminus E\), choose finitely many cross-cuts connecting the
boundary components to \(\partial D(0,R)\), and include also their reflected
images, so that the remaining domain is simply connected. Denote this remaining domain by
	\(U_1\). Let
		\(\Gamma\) be the union of the cuts, the circle \(\partial D(0,R)\), and the line  $[R,\infty)$. It is a finite union of curves and hence has planar measure
		zero. The stated conclusions follow.
	\end{proof}
	
	If \(\Delta\equiv\Delta_0\ne0\), we use the convention
	\begin{equation}\label{eq:composition-constant-Delta-convention}
		E=\Gamma=\varnothing,
		\qquad
		U_0=\mathbb C_I,
		\qquad
		U_1=\varnothing,
		\qquad
		\varepsilon_0=|\Delta_0|.
	\end{equation}
	Choose a constant square root of \(\Delta_0\). Whenever \(U_j=\varnothing\), all sums, integrals, and measures
	indexed by \(j\) are omitted.
	
	For each nonempty \(U_j\), choose
	\[
	\rho_j\in\mathcal O(U_j),
	\qquad \rho_j^2=\Delta,
	\]
	and set
	\[
	\lambda_{\pm,j}:=\frac{\tau\pm\rho_j}{2}.
	\]
	Let \(c_{\pm,j},d_{\pm,j}\) be the coefficients in
	\eqref{eq:composition-local-coefficients}. For a Borel set
	\(\mathcal A\subset\mathbb C_I\), define
	\begin{align*}
		\nu_j^{c,\pm}(\mathcal A)
		&:=
		\int_{U_j\cap\lambda_{\pm,j}^{-1}(\mathcal A)}
		|c_{\pm,j}(z)|^q
		\mathrm e^{-\frac{\alpha q}{2}|z|^2}
		\,\mathrm d m_{2,I}(z),\\
		\nu_j^{d,\pm}(\mathcal A)
		&:=
		\int_{U_j\cap\lambda_{\pm,j}^{-1}(\mathcal A)}
		|d_{\pm,j}(z)|^q
		\mathrm e^{-\frac{\alpha q}{2}|z|^2}
		\,\mathrm d m_{2,I}(z).
	\end{align*}
	The associated weighted measures are
	\begin{equation}\label{eq:composition-Lambda-def}
		\mathrm d\Lambda_j^{c,\pm}(\zeta)
		:=
		\mathrm e^{\frac{\alpha q}{2}|\zeta|^2}
		\,\mathrm d\nu_j^{c,\pm}(\zeta),
		\qquad
		\mathrm d\Lambda_j^{d,\pm}(\zeta)
		:=
		\mathrm e^{\frac{\alpha q}{2}|\zeta|^2}
		\,\mathrm d\nu_j^{d,\pm}(\zeta).
		\nonumber
	\end{equation}
	For \(w\in\mathbb C_I\), define
	\begin{align}
		\mathcal B_{U_j}(w)
		&:=
		\int_{U_j}
		|(k_w\bullet\psi)_I(z)|^q
		\mathrm e^{-\frac{\alpha q}{2}|z|^2}
		\,\mathrm d m_{2,I}(z),\nonumber
	 \\
		\mathcal B_E(w)
		&:=
		\int_E
		|(k_w\bullet\psi)_I(z)|^q
		\mathrm e^{-\frac{\alpha q}{2}|z|^2}
		\,\mathrm d m_{2,I}(z), 
		\nonumber
	\end{align}
	where \(\mathcal B_E\equiv0\) under the convention
	\eqref{eq:composition-constant-Delta-convention}. Since
	\(m_{2,I}(\Gamma)=0\),
	\begin{equation}\label{eq:composition-B-decomposition}
		\mathcal B_{\psi,I}
		=
		\mathcal B_{U_0}+\mathcal B_{U_1}+\mathcal B_E.
	\end{equation}

	\subsection{Carleson estimates on the good regions}
	\label{subsec:composition-good-Carleson}
	 
	\begin{proposition}\label{prop:composition-good-p-le-q}
		Assume \(Q\not\equiv0\), \(0<p\le q<\infty\), and
		\(\mathfrak B_\infty<\infty\). Use the good/singular notation above. Then,
		for \(j=0,1\), the measures
		\[
		\Lambda_j^{c,\pm},
		\qquad
		\Lambda_j^{d,\pm}
		\]
		are \((p,q)\)-Fock--Carleson measures, and their Carleson embedding constants
		are bounded by \(C\mathfrak B_\infty^q\).
		
		If, in addition,
		\[
		\mathcal B_{\psi,I}(w)\to0
		\qquad (|w|\to\infty),
		\]
		then all these measures are vanishing \((p,q)\)-Fock--Carleson measures.
	\end{proposition}
	
	\begin{proof}
		On \(U_j\), the two local eigenvalue branches are uniformly separated:
		\[
		|\lambda_{+,j}(z)-\lambda_{-,j}(z)|^2
		=|\Delta(z)|\ge\varepsilon_0.
		\]
		Thus \eqref{eq:composition-separated-c-infty} and
		\eqref{eq:composition-separated-d-infty} imply
		\begin{equation}\label{eq:composition-good-density-bound}
			M_z(c_{\pm,j},\lambda_{\pm,j})
			+
			M_z(d_{\pm,j},\lambda_{\pm,j})
			\lesssim
			\mathfrak B_\infty.
			\nonumber
		\end{equation}
		Consequently,
		\begin{equation}\label{eq:composition-good-Berezin-bound}
			\begin{aligned}
				\widetilde\Lambda_j^{c,\pm}(w)
				&=
				\int_{U_j}
				M_z(c_{\pm,j},\lambda_{\pm,j})^q
				\exp\!\left(-\frac{\alpha q}{2}
				|w-\lambda_{\pm,j}(z)|^2\right)
				\,\mathrm d m_{2,I}(z)\\
				&\lesssim
				\mathfrak B_\infty^q
				\int_{U_j}
				\exp\!\left(-\frac{\alpha q}{2}
				|w-\lambda_{\pm,j}(z)|^2\right)
				\,\mathrm d m_{2,I}(z).
			\end{aligned}
			\nonumber
		\end{equation}
		The same estimate holds for \(\widetilde\Lambda_j^{d,\pm}\).
		
		The integral over the bounded region \(U_1\) is uniformly bounded. Suppose
		first that \(\Delta\) is nonconstant. On \(U_0\),
		by \eqref{eq:composition-branch-affine-asymptotics}, after increasing the exterior radius if necessary,
		\[
		|w-\lambda_{\pm,0}(z)|^2
		\ge
		\frac12|w-a_\pm z-b_\pm|^2-1.
		\]
		An affine change of variables therefore yields
		\begin{equation}\label{eq:composition-Gaussian-pullback}
			\sup_{w\in\mathbb C_I}
			\int_{U_0}
			\exp\!\left(-\frac{\alpha q}{2}
			|w-\lambda_{\pm,0}(z)|^2\right)
			\,\mathrm d m_{2,I}(z)
			<\infty.
		\end{equation}
		
		If \(\Delta\equiv\Delta_0\ne0\), then
		\begin{equation}\label{eq:composition-constant-discriminant-branches}
		\lambda_{\pm,0}(z)
		=
		\frac{t_1}{2}z+\frac{t_0\pm\rho}{2}.
		\end{equation} 
		When \(t_1\ne0\), the same affine Gaussian estimate applies. When
		\(t_1=0\), the two eigenvalue branches are distinct constants. Choose
		\(w_1,w_2\in\mathbb C_I\) such that
		\[
		\det
		\begin{pmatrix}
			k_{w_1}(\lambda_{+,0})&k_{w_1}(\lambda_{-,0})\\[2pt]
			k_{w_2}(\lambda_{+,0})&k_{w_2}(\lambda_{-,0})
		\end{pmatrix}
		\ne0.
		\]
		The identities
		\[
		\bigl(\Phi_I(k_{w_i}\bullet\psi)\bigr)_{11}
		=
		c_{+,0}k_{w_i}(\lambda_{+,0})
		+c_{-,0}k_{w_i}(\lambda_{-,0}),
		\qquad i=1,2,
		\]
		form a fixed invertible linear system. Hence
		\[
		\int_{\mathbb C_I}
		\bigl(|c_{+,0}|^q+|c_{-,0}|^q\bigr)
		\mathrm e^{-\frac{\alpha q}{2}|z|^2}
		\,\mathrm d m_{2,I}(z)
		\lesssim
		\mathfrak B_\infty^q.
		\]
		Thus \(\Lambda_0^{c,\pm}\) is a finite point mass. Applying the same argument
		to the \((1,2)\)-entry  gives the conclusion
		for \(\Lambda_0^{d,\pm}\).
		
		We have therefore shown that all the relevant Berezin transforms are bounded.
	The complex Fock--Carleson theorem
\cite[Theorem~2.6]{MR3248473} gives the first assertion; see also
\cite{MR2609242,MR2964691,MR2819157} for related Fock--Carleson measure
characterizations.
		
	Now assume that
	\[
	\mathcal B_{\psi,I}(w)\longrightarrow0
	\qquad (|w|\to\infty).
	\]
	Evaluating the local kernel estimates at
	\(w=\lambda_{+,j}(z)\) and \(w=\lambda_{-,j}(z)\), and using the
	uniform separation of the two branches on \(U_j\), we obtain
	\begin{equation}\label{eq:composition-branch-density-vanishing-control}
		\begin{aligned}
			&M_z(c_{\pm,j},\lambda_{\pm,j})^q
			+
			M_z(d_{\pm,j},\lambda_{\pm,j})^q \lesssim
			\sum_{\sigma=\pm}
			\left[
			\mathcal B_{\psi,I}(\lambda_{\sigma,j}(z))
			+
			\mathcal B_{\psi,I}
			(\overline{\lambda_{\sigma,j}(z)})
			\right].
		\end{aligned}
	\end{equation}
	
	Consider first the unbounded region \(U_0\), and suppose that the
	branches are nonconstant. Their affine asymptotics imply that
	\[
	|\lambda_{\sigma,0}(z)|\longrightarrow\infty,
	\qquad \sigma\in\{+,-\},
	\]
	as \(|z|\to\infty\) in \(U_0\). Hence
	\eqref{eq:composition-branch-density-vanishing-control} gives
	\[
	\sup_{\substack{z\in U_0\\ |z|>R}}
	\left[
	M_z(c_{\pm,0},\lambda_{\pm,0})^q
	+
	M_z(d_{\pm,0},\lambda_{\pm,0})^q
	\right]
	\longrightarrow0
	\qquad (R\to\infty).
	\]
	
	Fix \(R>0\) and decompose, for example,
	\[
	\widetilde\Lambda_0^{c,\pm}(w)
	=
	\int_{U_0\cap D(0,R)}
	M_z(c_{\pm,0},\lambda_{\pm,0})^q
	\mathrm e^{-\frac{\alpha q}{2}
		|w-\lambda_{\pm,0}(z)|^2}
	\,\mathrm d m_{2,I}(z)
	\]
	\[
	\hspace{2cm}
	+
	\int_{U_0\setminus D(0,R)}
	M_z(c_{\pm,0},\lambda_{\pm,0})^q
	\mathrm e^{-\frac{\alpha q}{2}
		|w-\lambda_{\pm,0}(z)|^2}
	\,\mathrm d m_{2,I}(z).
	\]
	On the bounded part \(U_0\cap D(0,R)\), both the branch density and
	\(\lambda_{\pm,0}\) are bounded. Therefore the Gaussian factor tends
	to zero uniformly there as \(|w|\to\infty\), and the first integral
	tends to zero. On the exterior part, the preceding density estimate
	and \eqref{eq:composition-Gaussian-pullback} give
	\[
	\begin{aligned}
		&\int_{U_0\setminus D(0,R)}
		M_z(c_{\pm,0},\lambda_{\pm,0})^q
		\mathrm e^{-\frac{\alpha q}{2}
			|w-\lambda_{\pm,0}(z)|^2}
		\,\mathrm d m_{2,I}(z)  \lesssim
		\sup_{\substack{z\in U_0\\ |z|>R}}
		M_z(c_{\pm,0},\lambda_{\pm,0})^q,
	\end{aligned}
	\]
	uniformly in \(w\). Letting first \(|w|\to\infty\) and then
	\(R\to\infty\), we obtain $
	\widetilde\Lambda_0^{c,\pm}(w)\longrightarrow0$. 
	The same argument applies to
	\(\widetilde\Lambda_0^{d,\pm}\).
	
	On the bounded region \(U_1\), the branch functions and branch
	densities are bounded. Hence the Gaussian factor tends uniformly to
	zero as \(|w|\to\infty\), and consequently
	\[
	\widetilde\Lambda_1^{c,\pm}(w)\longrightarrow0,
	\qquad
	\widetilde\Lambda_1^{d,\pm}(w)\longrightarrow0.
	\]
If the two eigenvalue branches are constant, the corresponding measures are
finite point masses, whose Berezin transforms also tend to zero. 
	The vanishing part of the complex Fock--Carleson theorem completes
	the proof.
	\end{proof}
	
	\begin{proposition}\label{prop:composition-good-q-less-p}
	Under the standing notation of the two-branch case, let 
\(0<q<p<\infty\) and put \(s=p/(p-q)\). Assume that the 
good/singular decomposition is available; this is the case, for instance, 
under either \(\mathfrak B_\infty<\infty\) or \(\mathfrak B_s<\infty\). 
Then the following assertions hold.
		
		\smallskip
		\noindent\textup{(i)} If
		\[
		\mathcal B_{\psi,I}\in L^s(\mathbb C_I),
		\]
		then each of the measures
		\[
		\Lambda_j^{c,\pm},
		\qquad
		\Lambda_j^{d,\pm},
		\qquad j=0,1,
		\]
		is a vanishing \((p,q)\)-Fock--Carleson measure. Moreover, their
		Carleson embedding constants are controlled by
		\(C\|\mathcal B_{\psi,I}\|_{L^s}\).
		
		\smallskip
		\noindent\textup{(ii)} If \(C_\psi^\bullet:F_\alpha^p(\mathbb H)\to
		F_\alpha^q(\mathbb H)\) is bounded, then
		\[
		\mathcal B_{U_j}\in L^s(\mathbb C_I),
		\qquad j=0,1,
		\]
		and
		\[
		\|\mathcal B_{U_j}\|_{L^s}^{1/q}
		\lesssim
		\|C_\psi^\bullet\|.
		\]
	\end{proposition}
	
	\begin{proof}
		We first prove \textup{(i)}. Since
		\(\mathcal B_{U_j}\le\mathcal B_{\psi,I}\), one has
		\(\mathcal B_{U_j}\in L^s\). Choose \(r>0\) so small that
		\(3r<\sqrt{\varepsilon_0}\). We prove that
		\begin{equation}\label{eq:composition-Lambda-disk-Ls}
			\Lambda_j^{c,\pm}(D(\cdot,r))
			\in L^s(\mathbb C_I).
		\end{equation}
		The proof for the  measures $	\Lambda_j^{d,\pm}$ is identical.
		
		For the bounded region \(U_1\), the coefficient functions and the local eigenvalue
		branches are bounded, so
		\eqref{eq:composition-Lambda-disk-Ls} is immediate. On \(U_0\), assume first
		that \(\Delta\) is nonconstant. By \eqref{eq:composition-branch-affine-asymptotics}, after enlarging the exterior radius,  
		\[
		|\lambda_{\pm,0}(z)-a_\pm z-b_\pm|\le r/2.
		\]
		For
		\[
		u_\pm(w):=\frac{w-b_\pm}{a_\pm},
		\qquad
		E_w^\pm:=U_0\cap\lambda_{\pm,0}^{-1}(D(w,r)),
		\]
		we have
		\begin{equation}\label{eq:composition-preimage-disk-comparison}
			U_0\cap D\!\left(u_\pm(w),\frac{r}{2|a_\pm|}\right)
			\subset E_w^\pm
			\subset
		U_0\cap 	D\!\left(u_\pm(w),\frac{3r}{2|a_\pm|}\right).
		\end{equation}
		On \(E_w^+\),
		\[
		|\lambda_{+,0}(z)-w|<r,
		\qquad
		|\lambda_{-,0}(z)-w|>2r.
		\]
		Thus
		\begin{equation}\label{eq:composition-difference-lower-plus}
			\begin{aligned}
				&\left|
				\sum_{\sigma=\pm}
				c_{\sigma,0}(z)k_w(\lambda_{\sigma,0}(z))
				\right|
				\mathrm e^{-\frac\alpha2|z|^2} \ge
				\mathrm e^{-\frac{\alpha r^2}{2}}
				M_z(c_{+,0},\lambda_{+,0})
				-
				\mathrm e^{-2\alpha r^2}
				M_z(c_{-,0},\lambda_{-,0}).
			\end{aligned}
			\nonumber
		\end{equation}
		The analogous estimate with the signs interchanged holds on \(E_w^-\).
		Since \(\mathrm e^{-\alpha r^2/2}>\mathrm e^{-2\alpha r^2}\), the elementary
		inequality
		\[
		x^q+y^q
		\lesssim
		(Ax-By)_+^q+(Ay-Bx)_+^q,
		\qquad x,y\ge0,
		\quad A>B>0,
		\]
		combined with \eqref{eq:composition-preimage-disk-comparison}, the affine
		changes of variables \(w\mapsto u_\pm(w)\), and
		\(\mathcal B_{U_0}\in L^s\), gives
		\[
		u\longmapsto
		\int_{D(u,\delta)\cap U_0}
		M_z(c_{\pm,0},\lambda_{\pm,0})^q
		\,\mathrm d m_{2,I}(z)
		\in L^s(\mathbb C_I)
		\]
		for some, hence every, fixed \(\delta>0\). The outer inclusion in
		\eqref{eq:composition-preimage-disk-comparison} then gives
		\eqref{eq:composition-Lambda-disk-Ls}.
		
	If \(\Delta\equiv\Delta_0\ne0\) and the common slope of the two eigenvalue branches is
	nonzero, the same argument applies with exact affine maps. If both eigenvalue branches are constant, choose two kernel parameters outside the null set on which
		\(\mathcal B_{U_0}< \infty\), and solve the same fixed linear system used in
		the proof of Proposition~\ref{prop:composition-good-p-le-q}. This shows that
		the measures are finite point masses and hence satisfy
		\eqref{eq:composition-Lambda-disk-Ls}.
		
		By the complex criterion for \(0<q<p<\infty\),
		see \cite[Theorem~2.6]{MR3248473}, condition
		\eqref{eq:composition-Lambda-disk-Ls} is equivalent to the
		\((p,q)\)-Fock--Carleson measure. In this range boundedness and compactness of
		the embedding are equivalent, so the measures are vanishing. The estimates
		above also give the asserted quantitative control.
 
		\smallskip
		\noindent\textup{(ii)}
		Assume that \(C_\psi^\bullet\) is bounded. The kernel test gives $
		\mathfrak B_\infty
		\lesssim
		\|C_\psi^\bullet\|$. 
Thus Theorem~\ref{thm:composition-spectral-rigidity} applies, and the
good/singular decomposition of Lemma~\ref{lem:composition-good-singular-decomposition}
is available. On the bounded
		region \(U_1\), they imply directly that
		\[
		\Lambda_1^{c,\pm}(D(\cdot,r)),
		\quad
		\Lambda_1^{d,\pm}(D(\cdot,r))
		\in L^s(\mathbb C_I),
		\]
		with norms controlled by \(C\|C_\psi^\bullet\|^q\).
		
		It remains to treat \(U_0\). Let \(\{a_m\}\) be an \(r\)-lattice in \(\mathbb C_I\), see \cite[Section~2]{MR2934601}.  Let
		\(\{\beta_m\}\in\ell^p\), and put
		\[
		H_t(\zeta)
		:=
		\sum_m r_m(t)\beta_m k_{a_m}(\zeta).
		\]
		Let \(f_t\) be the slice regular extension satisfying
		\((f_t)_I=H_t\). The boundedness of
		\(C_\psi^\bullet\), Fubini's theorem, and Khinchine's inequality give
		\begin{equation}\label{eq:composition-Khinchine-good}
			\begin{aligned}
				&\int_{U_0}
				\left(
				\sum_m|\beta_m|^2
				\left|
				\sum_{\sigma=\pm}
				c_{\sigma,0}(z)
				k_{a_m}(\lambda_{\sigma,0}(z))
				\right|^2
				\right)^{q/2}
				\mathrm e^{-\frac{\alpha q}{2}|z|^2}
				\,\mathrm d m_{2,I}(z)\\
				&\qquad\lesssim
				\|C_\psi^\bullet\|^q
				\|\{\beta_m\}\|_{\ell^p}^q.
			\end{aligned}
		\end{equation}
		
		Set
		\[
		E_m^\pm
		:=
		U_0\cap\lambda_{\pm,0}^{-1}(D(a_m,r))
		\]
		and
		\[
		\alpha_m^{c,\pm}
		:=
		\int_{E_m^\pm}
		\left|
		\sum_{\sigma=\pm}
		c_{\sigma,0}(z)
		k_{a_m}(\lambda_{\sigma,0}(z))
		\right|^q
		\mathrm e^{-\frac{\alpha q}{2}|z|^2}
		\,\mathrm d m_{2,I}(z).
		\]
		The finite-overlap property of the lattice and
		\eqref{eq:composition-Khinchine-good} yield
		\[
		\sum_m|\beta_m|^q\alpha_m^{c,\pm}
		\lesssim
		\|C_\psi^\bullet\|^q
		\|\{\beta_m\}\|_{\ell^p}^q.
		\]
		By duality between \(\ell^{p/q}\) and \(\ell^s\),
		\[
		\{\alpha_m^{c,\pm}\}_m\in\ell^s,
		\qquad
		\|\{\alpha_m^{c,\pm}\}\|_{\ell^s}^{1/q}
		\lesssim
		\|C_\psi^\bullet\|.
		\]
		
	The remainder is the standard argument used in the proof of
	part \textup{(i)} and in
	\cite[Lemma~2.3 and the proof of Theorem~A]{MR4811250}. Namely, the
	lattice masses of the push-forward measures with densities
	\[
	\left(
	\mathrm e^{-\frac{\alpha r^2}{2}}
	M_z(c_{\pm,0},\lambda_{\pm,0})
	-
	\mathrm e^{-2\alpha r^2}
	M_z(c_{\mp,0},\lambda_{\mp,0})
	\right)_+^q
	\]
	are controlled by \(\alpha_m^{c,\pm}\). Lemma~2.3 of
	\cite{MR4811250} therefore gives their fixed-radius local mass
	functions in \(L^s\). Combining the two signs and using
	\eqref{eq:composition-preimage-disk-comparison} yields
	\[ 
	\|\Lambda_0^{c,\pm}(D(\cdot,r))\|_{L^s}^{1/q}
	\lesssim
	\|C_\psi^\bullet\|.
	\]
	The globally affine and constant-eigenvalue-branch cases are treated as in
	part \textup{(i)}.  Repeating the argument with \(F=0\) and the
		reflected lattice gives the corresponding estimates for
		\(\Lambda_j^{d,\pm}\).
		
		By \cite[Theorem~2.7]{MR3248473}, the Berezin transforms of all these
		measures belong to \(L^s(\mathbb C_I)\), with the same norm control.
		The local spectral formula gives, for \(j=0,1\),
		\[
		\mathcal B_{U_j}(w)
		\lesssim
		\sum_{\sigma=\pm}
		\left(
		\widetilde\Lambda_j^{c,\sigma}(w)
		+
		\widetilde\Lambda_j^{d,\sigma}(\overline w)
		\right).
		\]
		Since reflection preserves planar measure, it follows that
		\[
		\|\mathcal B_{U_j}\|_{L^s}^{1/q}
		\lesssim
		\|C_\psi^\bullet\|,
		\qquad j=0,1.
		\]
	\end{proof}

\subsection{Singular-region estimates}
\label{subsec:composition-singular-region}

\begin{proposition} 
	\label{prop:composition-singular-estimates}
	Use the good/singular notation introduced above.
	
	\smallskip
	\noindent\textup{(i)}
	If \(\Delta\) is nonconstant, then, for every \(\varepsilon>0\),
	the singular neighborhood \(E\) may be chosen so that
	\[
	\int_E
	|(f\bullet\psi)_I(z)|^q
	\mathrm e^{-\frac{\alpha q}{2}|z|^2}
	\,\mathrm d m_{2,I}(z)
	\le
	\varepsilon
	\|f\|_{F_\alpha^p(\mathbb H)}^q
	\]
	for every \(f\in F_\alpha^p(\mathbb H)\).
	
	\smallskip
	\noindent\textup{(ii)}
	Let \(0<q<p<\infty\), and put \(s=p/(p-q)\). For every fixed
	choice of \(E\),
	\[
	\mathcal B_E\in L^s(\mathbb C_I).
	\]
	Moreover,
	\begin{equation}\label{eq:composition-singular-control-by-BE}
		\int_E
		|(f\bullet\psi)_I(z)|^q
		\mathrm e^{-\frac{\alpha q}{2}|z|^2}
		\,\mathrm d m_{2,I}(z)
		\lesssim
		\|\mathcal B_E\|_{L^s}
		\|f\|_{F_\alpha^p(\mathbb H)}^q
	\end{equation}
	for every \(f\in F_\alpha^p(\mathbb H)\). If \(C_\psi^\bullet\)
	is bounded, then
	\begin{equation}\label{eq:composition-BE-operator-bound}
		\|\mathcal B_E\|_{L^s}^{1/q}
		\lesssim
		\|C_\psi^\bullet\|.
	\end{equation}
\end{proposition}

\begin{proof}
	If \(\Delta\) is constant, then \(E=\varnothing\) under
	\eqref{eq:composition-constant-Delta-convention}, and there is
	nothing to prove. Assume therefore that \(\Delta\) is nonconstant.
	
	Choose \(R_0>0\) such that
	\[
	Z(\Delta)\subset D(0,R_0)
	\]
	and take \(E\subset\overline{D(0,R_0)}\). By
	\eqref{eq:composition-compact-set-estimate},
	\[
	\sup_{z\in\overline{D(0,R_0)}}
	\|\Phi_I(f\bullet\psi)(z)\|
	\lesssim
	\|f\|_{F_\alpha^p(\mathbb H)}.
	\]
	Hence \eqref{eq:composition-first-row-equivalence} gives
	\[
	\int_E
	|(f\bullet\psi)_I(z)|^q
	\mathrm e^{-\frac{\alpha q}{2}|z|^2}
	\,\mathrm d m_{2,I}(z)
	\lesssim
	m_{2,I}(E)
	\|f\|_{F_\alpha^p(\mathbb H)}^q.
	\]
	Since the set \(E\) may be chosen with arbitrarily small
	 planar measure, part \textup{(i)} follows.
	
	We now prove part \textup{(ii)}. Put
	\[
	M_E:=\sup_{z\in E}\|\Psi(z)\|<\infty .
	\]
	By the   formula
	\eqref{eq:composition-kernel-matrix-exponential}, we have, for \(z\in E\),
	\[
	\left|\bigl(\Phi_I(k_w\bullet\psi)\bigr)_{11}(z)\right|
	+
	\left|\bigl(\Phi_I(k_w\bullet\psi)\bigr)_{12}(z)\right|
	\lesssim
	\mathrm e^{-\frac{\alpha}{2}|w|^2+\alpha M_E|w|}.
	\]
	Since \(E=E^\#\), we have
	\[
	\sup_{z\in E}
	\|\Phi_I(k_w\bullet\psi)(z)\|
	\lesssim_E
	\mathrm e^{-\frac{\alpha}{2}|w|^2+\alpha M_E|w|}.
	\]
	Since \(k_0\equiv1\), we have
	\[
	\mathcal B_E(0)
	=
	\int_E
	\mathrm e^{-\frac{\alpha q}{2}|z|^2}
	\,\mathrm d m_{2,I}(z).
	\]
	It follows that
	\begin{equation}\label{eq:composition-BE-Gaussian}
		\mathcal B_E(w)
		\lesssim 
		\mathcal B_E(0)
		\mathrm e^{-\frac{\alpha q}{2}|w|^2+\alpha qM_E|w|}
		\lesssim 
		\mathcal B_E(0)\mathrm e^{-c|w|^2},
	\end{equation}
	for some \(c>0\). Consequently,
	\[
	\mathcal B_E\in L^s(\mathbb C_I),
	\qquad
	\|\mathcal B_E\|_{L^s}
	\lesssim
	\mathcal B_E(0).
	\]
	
	Conversely, by \eqref{eq:composition-kernel-matrix-exponential},
	\[
	\Phi_I(k_w\bullet\psi)\longrightarrow \mathrm I_2
	\]
	uniformly on \(E\) as \(w\to0\). Hence there exists \(\delta>0\) such that
	\[
	|(k_w\bullet\psi)_I(z)|\ge\frac12,
	\qquad z\in E,\quad |w|<\delta.
	\]
	Therefore
	\[
	\mathcal B_E(w)
	\ge
	2^{-q}\mathcal B_E(0),
	\qquad |w|<\delta,
	\]
	and hence
	\begin{equation}\label{eq:composition-weighted-area-by-BE}
		\mathcal B_E(0)
		\lesssim
		\|\mathcal B_E\|_{L^s}.
	\end{equation}
	
	Now let \(f\in F_\alpha^p(\mathbb H)\). By
	\eqref{eq:composition-compact-set-estimate} and
	\eqref{eq:composition-first-row-equivalence},
	\[
	\begin{aligned}
		&\int_E
		|(f\bullet\psi)_I(z)|^q
		\mathrm e^{-\frac{\alpha q}{2}|z|^2}
		\,\mathrm d m_{2,I}(z)  \lesssim 
		\|f\|_{F_\alpha^p(\mathbb H)}^q
		\int_E
		\mathrm e^{-\frac{\alpha q}{2}|z|^2}
		\,\mathrm d m_{2,I}(z)  =
		\mathcal B_E(0)\,
		\|f\|_{F_\alpha^p(\mathbb H)}^q .
	\end{aligned}
	\]
	Combining this with \eqref{eq:composition-weighted-area-by-BE} gives
	\[
	\int_E
	|(f\bullet\psi)_I(z)|^q
	\mathrm e^{-\frac{\alpha q}{2}|z|^2}
	\,\mathrm d m_{2,I}(z)
	\lesssim
	\|\mathcal B_E\|_{L^s}
	\|f\|_{F_\alpha^p(\mathbb H)}^q,
	\]
	which is \eqref{eq:composition-singular-control-by-BE}.
	
	Finally, if \(C_\psi^\bullet\) is bounded, then
	\[
	\mathcal B_E(0)
	\le
	\mathcal B_{\psi,I}(0)
	\approx
	\|C_\psi^\bullet k_0\|_{F_\alpha^q(\mathbb H)}^q
	\le
	\|C_\psi^\bullet\|^q
	\|k_0\|_{F_\alpha^p(\mathbb H)}^q
	\lesssim
	\|C_\psi^\bullet\|^q.
	\]
	Together with the first consequence of
	\eqref{eq:composition-BE-Gaussian}, namely
	\[
	\|\mathcal B_E\|_{L^s}
	\lesssim
	\mathcal B_E(0),
	\]
	this yields
	\[
	\|\mathcal B_E\|_{L^s}^{1/q}
	\lesssim
	\|C_\psi^\bullet\|,
	\]
	as claimed in \eqref{eq:composition-BE-operator-bound}.
\end{proof}

	\subsection{Proofs of the main theorems}
	\label{subsec:composition-main-proofs}
	
	\begin{proof}[Proof of Theorem~\ref{thm:composition-p-le-q}]
		Fix \(I\in\mathbb S\), and choose \(J\perp I\). Write $\psi_I=P+QJ$, where $P,Q\in \mathcal{O}(\mathbb{C}_I)$.
		
		\smallskip
		\noindent\emph{Necessity of the boundedness condition.}
		For \(w\in\mathbb C_I\), the normalized kernels satisfy
		\(\|k_w\|_{F_\alpha^p(\mathbb H)}\approx1\). By the definition of the Fock
		norm,
		\begin{equation}\label{eq:composition-kernel-norm-test}
			\mathcal B_{\psi,I}(w)
			\approx
			\|C_\psi^\bullet k_w\|_{F_\alpha^q(\mathbb H)}^q.
		\end{equation}
		Thus boundedness implies
		\[
		\sup_w\mathcal B_{\psi,I}(w)^{1/q}
		\lesssim
		\|C_\psi^\bullet\|.
		\]
		
	\smallskip
	\noindent\emph{Sufficiency.}
	Assume \(\mathfrak B_\infty<\infty\). The case \(Q\equiv0\) was
	settled in Subsection~\ref{subsec:composition-scalar-reductions}.
	Assume henceforth that \(Q\not\equiv0\).
	
	By Theorem~\ref{thm:composition-spectral-rigidity}, \(\Delta\) is
	either a nonzero constant or a quadratic polynomial with negative
	leading coefficient. If \(\Delta\) is nonconstant, choose the singular
	neighborhood \(E\) in
Proposition~\ref{prop:composition-singular-estimates}\textup{(i)} so that
	\[
	\int_E
	|(f\bullet\psi)_I(z)|^q
	\mathrm e^{-\frac{\alpha q}{2}|z|^2}
	\,\mathrm d m_{2,I}(z)
	\le
\mathfrak  B_\infty^q
	\|f\|_{F_\alpha^p(\mathbb H)}^q.
	\]
	Complete \(E\) to the good/singular decomposition of
	Lemma~\ref{lem:composition-good-singular-decomposition}. In the
	constant-discriminant case, use the convention
	\eqref{eq:composition-constant-Delta-convention}.
	
	By Proposition~\ref{prop:composition-good-p-le-q}, all good-region
	measures are \((p,q)\)-Fock--Carleson measures, with embedding
	constants controlled by \(C\mathfrak B_\infty^q\). Hence
	\begin{equation}\label{eq:composition-good-operator-bound-p-le-q}
		\int_{U_0\cup U_1}
		|(f\bullet\psi)_I(z)|^q
		\mathrm e^{-\frac{\alpha q}{2}|z|^2}
		\,\mathrm d m_{2,I}(z)
		\lesssim
		\mathfrak B_\infty^q
		\|f\|_{F_\alpha^p(\mathbb H)}^q.
		\nonumber
	\end{equation}
 
	Combining the good and singular estimates proves boundedness and the
	upper norm estimate. Together with the kernel-test lower estimate,
	this gives
	\[
	\|C_\psi^\bullet\|
	\approx
	\mathfrak B_\infty.
	\]
		
	\smallskip
	\noindent\emph{Compactness.}
	If \(C_\psi^\bullet\) is compact, then the family \(\{k_w\}\) is
	bounded in \(F_\alpha^p(\mathbb H)\) and converges to zero uniformly
	on compact subsets as \(|w|\to\infty\). Hence
	\eqref{eq:composition-kernel-norm-test} gives for every \(I\in\mathbb S\),
	\[
	\mathcal B_{\psi,I}(w)\longrightarrow0
	\qquad
	\text{as }|w|\to\infty,\quad w\in\mathbb C_I.
	\]
	
	Conversely, assume that for some \(I\in\mathbb S\),
	\[
	\mathcal B_{\psi,I}(w)\longrightarrow0
	\qquad
	\text{as }|w|\to\infty,\quad w\in\mathbb C_I.
	\]
	The case \(Q\equiv0\) follows from the scalar reduction in
	Subsection~\ref{subsec:composition-scalar-reductions}. Assume
	\(Q\not\equiv0\). Proposition~\ref{prop:composition-good-p-le-q}
	shows that all good-region measures are vanishing
	\((p,q)\)-Fock--Carleson measures.
	
	Let \(\{f_n\}\) be bounded in \(F_\alpha^p(\mathbb H)\) and converge
	to zero uniformly on compact subsets. The vanishing Carleson
	embeddings imply
	\[
	\int_{U_0\cup U_1}
	|(f_n\bullet\psi)_I(z)|^q
	\mathrm e^{-\frac{\alpha q}{2}|z|^2}
	\,\mathrm d m_{2,I}(z)
	\longrightarrow0.
	\]
	If \(E\ne\varnothing\), then
	\eqref{eq:composition-compact-open-continuity} gives
	\[
	\sup_{z\in E}
	\|\Phi_I(f_n\bullet\psi)(z)\|
	\longrightarrow0.
	\]
	Consequently,
	\[
	\int_E
	|(f_n\bullet\psi)_I(z)|^q
	\mathrm e^{-\frac{\alpha q}{2}|z|^2}
	\,\mathrm d m_{2,I}(z)
	\longrightarrow0.
	\]
	In the constant-discriminant case \(E=\varnothing\). Thus
	\[
	\|C_\psi^\bullet f_n\|_{F_\alpha^q(\mathbb H)}
	\longrightarrow0,
	\]
	and the compactness criterion yields that \(C_\psi^\bullet\) is
	compact.
		
		Since the argument may be carried out on any fixed slice, a condition holding
		on one slice implies boundedness or compactness, and hence the corresponding
		condition on every slice.
	\end{proof}
	
	\begin{proof}[Proof of Theorem~\ref{thm:composition-q-less-p}]
		Put \(s=p/(p-q)\). The implication compact \(\Rightarrow\) bounded is
		immediate.
		
	\smallskip
	\noindent\emph{The testing condition implies compactness.}
	Assume $
	\mathcal B_{\psi,I}\in L^s(\mathbb C_I)$. 
	The case \(Q\equiv0\) was settled in
	Subsection~\ref{subsec:composition-scalar-reductions}. Assume
	\(Q\not\equiv0\).
	
	Since \(\mathfrak B_s<\infty\), 
	Theorem~\ref{thm:composition-spectral-rigidity} and the
	good/singular decomposition apply. By
	Proposition~\ref{prop:composition-good-q-less-p}\textup{(i)}, all
	good-region measures are vanishing
	\((p,q)\)-Fock--Carleson measures.
	
	Let \(\{f_n\}\) be bounded in \(F_\alpha^p(\mathbb H)\) and converge
	to zero uniformly on compact subsets. Writing
\[
(f_n)_I=F_n+G_nJ,
\]
we see from \eqref{eq:composition-local-four-term-11} and
\eqref{eq:composition-local-four-term-12} that the good-region
contributions are finite sums of vanishing Carleson embeddings applied to
\[
F_n,\qquad G_n,\qquad F_n^\#,\qquad G_n^\#.
\]
	Hence
	\[
	\int_{U_0\cup U_1}
	|(f_n\bullet\psi)_I(z)|^q
	\mathrm e^{-\frac{\alpha q}{2}|z|^2}
	\,\mathrm d m_{2,I}(z)
	\longrightarrow0.
	\]
	By \eqref{eq:composition-compact-open-continuity},
	\[
	\Phi_I(f_n\bullet\psi)\longrightarrow0
	\quad\text{uniformly on }E,
	\]
	and therefore
	\[
	\int_E
	|(f_n\bullet\psi)_I(z)|^q
	\mathrm e^{-\frac{\alpha q}{2}|z|^2}
	\,\mathrm d m_{2,I}(z)
	\longrightarrow0.
	\]
	Thus
	\[
	\|C_\psi^\bullet f_n\|_{F_\alpha^q(\mathbb H)}
	\longrightarrow0,
	\]
	and \(C_\psi^\bullet\) is compact.
		
		The quantitative good-region estimates, together with
		Proposition~\ref{prop:composition-singular-estimates}\textup{(ii)} and
		\(\mathcal B_E\le\mathcal B_{\psi,I}\), also give
		\begin{equation}\label{eq:composition-upper-norm-q-less-p}
			\|C_\psi^\bullet\|
			\lesssim
			\|\mathcal B_{\psi,I}\|_{L^s}^{1/q}.
		\end{equation}
		
		\smallskip
		\noindent\emph{Boundedness implies the testing condition.}
		Assume \(C_\psi^\bullet\) is bounded. If \(Q\equiv0\), the scalar complex
		criterion gives \(\mathcal B_{\psi,I}\in L^s\).
		
		Assume \(Q\not\equiv0\). The kernel test gives
		\(\mathfrak B_\infty<\infty\), so
		Theorem~\ref{thm:composition-spectral-rigidity} applies. By
		Proposition~\ref{prop:composition-good-q-less-p}\textup{(ii)},
		\[
		\mathcal B_{U_0},\mathcal B_{U_1}\in L^s,
		\qquad
		\|\mathcal B_{U_j}\|_{L^s}^{1/q}
		\lesssim
		\|C_\psi^\bullet\|.
		\]
		Proposition~\ref{prop:composition-singular-estimates}\textup{(ii)}
		gives the same conclusion for \(\mathcal B_E\). By \eqref{eq:composition-B-decomposition},
		\[
		\mathcal B_{\psi,I}\in L^s,
		\qquad
		\|\mathcal B_{\psi,I}\|_{L^s}^{1/q}
		\lesssim
		\|C_\psi^\bullet\|.
		\]
		Together with \eqref{eq:composition-upper-norm-q-less-p}, this proves the norm
		equivalence and all three equivalences.
		
		As before, validity on one slice implies boundedness or compactness, and hence
		validity on every slice.
	\end{proof}
	
\subsection{Globally affine eigenvalue functions do not characterize boundedness}
	
	A fundamental distinction from the scalar complex theory should be
	emphasized. In the complex setting, boundedness of a composition operator is
	characterized by the affine form of its inducing symbol; see
	\cite[Corollaries~3.5 and~3.6]{MR3905527}. No analogous characterization can
	be formulated solely in terms of the local eigenvalue branches of \(\Psi\). 
Even when the eigenvalue branches extend to globally affine contractions,
uncontrolled growth of the spectral coefficients
may prevent \(C_\psi^\bullet\) from being bounded.
	
\begin{example}\label{ex:affine-branches-not-sufficient}
	Fix \(I,J\in\mathbb S\) with \(J\perp I\). For \(m\in\{1,3\}\), let
	\(\psi_m\in\mathcal{SR}(\mathbb H)\) be the slice regular extension of
	\[
	(\psi_m)_I(z)
	=
	\frac12 z+I\cos(z^m)+\sin(z^m)J,
	\qquad z\in\mathbb C_I.
	\]
Write
	\[
	P_m(z)=\frac12z+I\cos(z^m),
	\qquad
	Q_m(z)=\sin(z^m).
	\]
Then the corresponding matrix symbol is
	\[
	\Phi (\psi_m)(z)
	=
	\begin{pmatrix}
		\frac12z+I\cos(z^m)&\sin(z^m)\\[2pt]
		-\sin(z^m)&\frac12z-I\cos(z^m)
	\end{pmatrix},
	\]
and
	\[
	\tau_m(z)=z,
	\qquad
	\Delta_m(z)
	=
	-4\cos^2(z^m)-4\sin^2(z^m)
	=-4.
	\]
	Choosing \(\rho=2I\), we obtain, for both \(m=1\) and \(m=3\),
	\[
	\lambda_\pm(z)=\frac12z\pm I.
	\]
In particular, the two symbols have exactly the same globally affine eigenvalue
functions, and both slopes have modulus strictly less than one. Their
	spectral coefficients are
	\[
	c_\pm^{(m)}(z)
	=
	\frac{1\pm\cos(z^m)}2,
	\qquad
	d_\pm^{(m)}(z)
	=
	\mp\frac I2\sin(z^m).
	\]
	
	For \(m=1\),
	\[
	|c_\pm^{(1)}(z)|+|d_\pm^{(1)}(z)|
	\lesssim
	\mathrm e^{|z|},
	\]
	whereas
	\[
	|\lambda_\pm(z)|^2-|z|^2
	\le
	-\frac34|z|^2+C(1+|z|).
	\]
	Consequently, for some \(c_0>0\),
	\[
	M_z(c_\pm^{(1)},\lambda_\pm)
	+
	M_z(d_\pm^{(1)},\lambda_\pm)
	\lesssim
	\mathrm e^{-c_0|z|^2+C|z|}.
	\]
Using \eqref{eq:composition-local-four-term-11} and \eqref{eq:composition-local-four-term-12}, it follows 
that 
\[
\mathcal B_{\psi_1,I}(w)\to0
\qquad (|w|\to\infty),
\]
and
\[
\mathcal B_{\psi_1,I}\in L^s(\mathbb C_I)
\qquad\text{for every }s>0.
\]
By Theorems~\ref{thm:composition-p-le-q} and
\ref{thm:composition-q-less-p}, the operator $
C_{\psi_1}^\bullet:
F_\alpha^p(\mathbb H)\longrightarrow F_\alpha^q(\mathbb H)$ 
is compact for every \(0<p,q<\infty\).
	
	In contrast, \(C_{\psi_3}^\bullet\) is not bounded between any such pair of
	Fock spaces. Indeed, the polynomial \(f(\xi)=\xi\) belongs to
	\(F_\alpha^p(\mathbb H)\) for every \(p>0\). 
	If \(C_{\psi_3}^\bullet\) were bounded into \(F_\alpha^q(\mathbb H)\), then
	the splitting norm equivalence would imply
	\[
	Q_3(z)=\sin(z^3)\in F_\alpha^q(\mathbb C_I).
	\]
	However, the standard point-evaluation estimate on the complex Fock space
	would then give
	\[
	|\sin(z^3)|
	\mathrm e^{-\frac{\alpha}{2}|z|^2}
	\lesssim
	\|Q_3\|_{F_\alpha^q(\mathbb C_I)},
	\qquad z\in\mathbb C_I.
	\]
	Taking
	\[
	z_r=r\mathrm e^{I\pi/6},
	\qquad r>0,
	\]
	we have \(z_r^3=Ir^3\), and hence
	\[
	|\sin(z_r^3)|
	\mathrm e^{-\frac{\alpha}{2}|z_r|^2}
	=
	\sinh(r^3)\mathrm e^{-\frac{\alpha}{2}r^2}
	\longrightarrow\infty.
	\]
	This contradiction proves that \(C_{\psi_3}^\bullet\) is unbounded.
	
Thus two quaternionic symbols may have identical globally affine eigenvalue
functions, with strictly contractive slopes,  while inducing respectively a
	compact and an unbounded composition operator. The boundedness of
	\(C_\psi^\bullet\) therefore depends essentially on the spectral
	coefficients \(c_\pm,d_\pm\), and cannot be determined from the global eigenvalue functions alone.
\end{example}
\section{Applications to Weighted Composition and Volterra-Type Operators}
\label{sec:applications}
 \subsection{Weighted Composition Operators}
 \label{sec:weighted-composition-operators}
 
 In this section we study the weighted composition operator
\(W_{(u,\psi)}\) defined in
Definition~\ref{def:composition-weighted-composition}. 
 This case is obtained by the same spectral decomposition, with one
 additional feature: the first row of each spectral projection is replaced by
 the first row of \(\Phi_I(u)E_\pm\). We record the modifications needed for
 the proof.
 
 For $I\in\mathbb S$ and $w\in\mathbb C_I$, set
 \begin{equation}\label{eq:weighted-testing-quantity}
 	\mathcal B_{(u,\psi,I)}(w)
 	:=
 	\int_{\mathbb C_I}
 	\left|\bigl(u\star(k_w\bullet\psi)\bigr)_I(z)\right|^q
 	\mathrm e^{-\frac{\alpha q}{2}|z|^2}
 	\,\mathrm d m_{2,I}(z).
 \end{equation}
 Since $k_0\equiv1$,
 \begin{equation}\label{eq:weighted-testing-at-zero}
 	\mathcal B_{(u,\psi,I)}(0)
 	=
 	\int_{\mathbb C_I}|u_I(z)|^q
 	\mathrm e^{-\frac{\alpha q}{2}|z|^2}
 	\,\mathrm d m_{2,I}(z).
 	\nonumber
 \end{equation}
 Thus $\mathcal B_{(u,\psi,I)}\equiv0$ if and only if $u\equiv0$.
 
 \begin{theorem}\label{thm:weighted-p-le-q}
 	Let $u,\psi\in\mathcal{SR}(\mathbb H)$ and $0<p\le q<\infty$. Then:
 	
 	\smallskip
 	\noindent\textup{(i)} The operator $
 	W_{(u,\psi)}:F_\alpha^p(\mathbb H)\longrightarrow F_\alpha^q(\mathbb H)$ 
 	is bounded if and only if, for some, equivalently for every, $I\in\mathbb S$,
 	\[
 	\sup_{w\in\mathbb C_I}\mathcal B_{(u,\psi,I)}(w)<\infty.
 	\]
 	Moreover,
 	\[
 	\|W_{(u,\psi)}\|
 	\approx
 	\sup_{w\in\mathbb C_I}\mathcal B_{(u,\psi,I)}(w)^{1/q}.
 	\]
 	
 	\smallskip
 	\noindent\textup{(ii)} The operator $
 	W_{(u,\psi)}:F_\alpha^p(\mathbb H)\longrightarrow F_\alpha^q(\mathbb H)$  is compact if and only if,
 	for some, equivalently for every, $I\in\mathbb S$,
 	\[
 	\mathcal B_{(u,\psi,I)}(w)\longrightarrow0
 	\qquad (|w|\to\infty,\ w\in\mathbb C_I).
 	\]
 \end{theorem}
 
 \begin{theorem}\label{thm:weighted-q-less-p}
 	Let $u,\psi\in\mathcal{SR}(\mathbb H)$ and $0<q<p<\infty$. Put
 	$s=p/(p-q)$. Then the following are equivalent:
 	\begin{enumerate}[(i)]
 		\item $W_{(u,\psi)}:F_\alpha^p(\mathbb H)\to F_\alpha^q(\mathbb H)$ is bounded;
 		\item $W_{(u,\psi)}:F_\alpha^p(\mathbb H)\to F_\alpha^q(\mathbb H)$ is compact;
 		\item for some, equivalently for every, $I\in\mathbb S$,
 		\[
 		\mathcal B_{(u,\psi,I)}\in
 		L^s(\mathbb C_I,\mathrm d m_{2,I}).
 		\]
 	\end{enumerate}
 	Moreover,
 	\[
 	\|W_{(u,\psi)}\|
 	\approx
 	\left(
 	\int_{\mathbb C_I}
 	\mathcal B_{(u,\psi,I)}(w)^s
 	\,\mathrm d m_{2,I}(w)
 	\right)^{1/(qs)}.
 	\]
 \end{theorem}
 
\paragraph{The weighted local two-branch representation.}
 \label{subsec:weighted-local-change}
 
 Fix $I,J\in\mathbb S$ with $J\perp I$. Retain all notation associated with
 $\psi$ from Section~\ref{sec:composition-operators};  
Write
\[
u_I=A+DJ,
\qquad
\Phi_I(u)
=
\begin{pmatrix}
	A&D\\[2pt]
	-D^\#&A^\#
\end{pmatrix}.
\]
 On a nondegenerate local region $U$, choose $\rho^2=\Delta$, set
 $\lambda_\pm=(\tau\pm\rho)/2$, and let $E_\pm$ be the spectral projections
 from \eqref{eq:composition-spectral-projections}. Define
 \begin{equation}\label{eq:weighted-branch-coefficients}
 	c_\pm^{u}:=(\Phi_I(u)E_\pm)_{11},
 	\qquad
 	d_\pm^{u}:=(\Phi_I(u)E_\pm)_{12}.
 \end{equation} 
 Explicitly,
 \begin{equation}\label{eq:weighted-explicit-branch-coefficients}
 	\begin{aligned}
 		c_+^{u}
 		&=\frac{A(P-\lambda_-)-DQ^\#}{\rho},
 		&\qquad
 		c_-^{u}
 		&=-\frac{A(P-\lambda_+)-DQ^\#}{\rho},\\[2pt]
 		d_+^{u}
 		&=\frac{AQ+D(P^\#-\lambda_-)}{\rho},
 		&
 		d_-^{u}
 		&=-\frac{AQ+D(P^\#-\lambda_+)}{\rho}.
 	\end{aligned}
 	\nonumber
 \end{equation}
 Since $E_++E_-=\mathrm I_2$,
 \begin{equation}\label{eq:weighted-coefficient-sums}
 	c_+^{u}+c_-^{u}=A,
 	\qquad
 	d_+^{u}+d_-^{u}=D. 
 \end{equation}
 
 \begin{proposition}\label{prop:weighted-local-formula}
 	Let $f_I=F+GJ$. On every nondegenerate local region,
 	\begin{equation}\label{eq:weighted-local-formula-11}
 		\bigl(\Phi_I(W_{(u,\psi)}f)\bigr)_{11}
 		=
 		\sum_{\sigma=\pm}
 		\left[
 		c_\sigma^{u}F(\lambda_\sigma)
 		-d_\sigma^{u}G^\#(\lambda_\sigma)
 		\right], 
 	\end{equation}
 	and
 	\begin{equation}\label{eq:weighted-local-formula-12}
 		\bigl(\Phi_I(W_{(u,\psi)}f)\bigr)_{12}
 		=
 		\sum_{\sigma=\pm}
 		\left[
 		d_\sigma^{u}F^\#(\lambda_\sigma)
 		+c_\sigma^{u}G(\lambda_\sigma)
 		\right]. 
 	\end{equation}
 \end{proposition}
 
 \begin{proof}
Use the functional-calculus identity
$H(\Psi)=H(\lambda_+)E_++H(\lambda_-)E_-$ in
\eqref{eq:Xf-Yf-global}. Since
\[
\Phi_I(W_{(u,\psi)}f)=\Phi_I(u)\Phi_I(f\bullet\psi),
\]
we multiply from the left by $\Phi_I(u)$ and take the first-row entries.
 \end{proof}

\paragraph{Scalar and one-branch reductions.}
 \label{subsec:weighted-degeneracies}
 
If $Q\equiv0$, then $\psi$ preserves the fixed complex slice \(\mathbb C_I\) and
 \begin{equation}\label{eq:weighted-scalar-formula}
 	\begin{aligned}
 		\bigl(W_{(u,\psi)}f\bigr)_I
 		={}&
 		\left[A(F\circ P)-D(G\circ P)^\#\right] 
 	 +
 		\left[A(G\circ P)+D(F\circ P)^\#\right]J.
 	\end{aligned}
 	\nonumber
 \end{equation}
 Define
 \begin{equation}\label{eq:weighted-scalar-pullback}
 	\mu_{(u,P,I)}(\mathcal A)
 	:=
 	\int_{P^{-1}(\mathcal A)}
 	\bigl(|A(z)|^q+|D^\#(z)|^q\bigr)
 	\mathrm e^{-\frac{\alpha q}{2}|z|^2}
 	\,\mathrm d m_{2,I}(z),
 	\nonumber
 \end{equation}
 and
 \[
 \mathrm d\nu_{(u,P,I)}(\zeta)
 :=
 \mathrm e^{\frac{\alpha q}{2}|\zeta|^2}
 \,\mathrm d\mu_{(u,P,I)}(\zeta).
 \]
 Exactly as in the scalar reduction in
 Section~\ref{subsec:composition-scalar-reductions},
 \begin{equation}\label{eq:weighted-scalar-Berezin}
 	\mathcal B_{(u,\psi,I)}(w)
 	\approx
 	\widetilde\nu_{(u,P,I)}(w), 
 \end{equation}
 and the norm of $W_{(u,\psi)}f$ is controlled by the corresponding scalar
 Fock--Carleson embedding. Thus Theorems~\ref{thm:weighted-p-le-q} and
 \ref{thm:weighted-q-less-p} follow in this case from the classical complex
 theorems. We therefore assume below that $Q\not\equiv0$.
 
 We define the global entire functions
 \begin{equation}\label{eq:weighted-HL}
 	\begin{aligned}
 		\mathcal H_u
 		&:=-\left(A\left(P-\frac\tau2\right)-DQ^\#\right)^2\Delta
 		+\frac{A^2}{4}\Delta^2,\\[2pt]
 		\mathcal L_u
 		&:=-\left(AQ+D\left(P^\#-\frac\tau2\right)\right)^2\Delta
 		+\frac{D^2}{4}\Delta^2.
 	\end{aligned}
 	\nonumber
 \end{equation}
 A direct calculation gives, on a nondegenerate local
 region, 
 \begin{equation}\label{eq:weighted-HL-local}
 	\mathcal H_u=c_+^{u}c_-^{u}\rho^4,
 	\qquad
 	\mathcal L_u=d_+^{u}d_-^{u}\rho^4. 
 \end{equation}
 
 In particular, every two-branch formula and every local kernel estimate from
 Section~\ref{sec:composition-operators} remains valid after the simultaneous
 substitutions
 \begin{equation}\label{eq:weighted-global-substitution}
 	c_\pm\mapsto c_\pm^{u},
 	\qquad
 	d_\pm\mapsto d_\pm^{u},
 	\qquad
 	\mathcal B_{\psi,I}\mapsto\mathcal B_{(u,\psi,I)}.
 \end{equation}
 The only genuinely new issue is that some of the four weighted coefficients may
 vanish identically.
 \begin{proposition}\label{prop:weighted-mixed-degeneracy}
 	Assume $u\not\equiv0$ and $\Delta\not\equiv0$. If, on a nondegenerate local
 	region, either
 	\[
 	c_+^{u}\equiv0,\qquad d_-^{u}\equiv0, \qquad \text{or} \qquad
 	c_-^{u}\equiv0,\qquad d_+^{u}\equiv0,
 	\]
then $Q\equiv0$. Hence the symbol preserves the fixed complex slice
\(\mathbb C_I\).
 \end{proposition}
 
 \begin{proof}
 	We treat the first alternative. Put
 	\[
 	t:=P-\lambda_-=\lambda_+-P^\#.
 	\]
 	The two vanishing identities are equivalent to
 	\[
 	At=DQ^\#,
 	\qquad
 	AQ=Dt.
 	\]
 	If $A\equiv0$, then $D\not\equiv0$ and the first identity gives
 	$Q^\#\equiv0$. Assume $A\not\equiv0$. Multiplying the first identity by $t$
 	and the second by $Q^\#$ gives $t^2=QQ^\#$. Since
 	\[
 	t=\frac{P-P^\#+\rho}{2},
 	\qquad
 	QQ^\#=\frac{(P-P^\#)^2-\rho^2}{4},
 	\]
 	we obtain $\rho(P-P^\#+\rho)=0$. Thus $t=0$ and $QQ^\#=0$ on the local
 	region. The identity theorem implies $Q\equiv0$ on
 	$\mathbb C_I$.
 \end{proof}
 \begin{proposition}
 	\label{prop:weighted-same-branch-reduction}
 	Assume \(u\not\equiv0\) and \(\Delta\not\equiv0\). Suppose that on every
 	nondegenerate local region \(U\), after possibly replacing the chosen square
 	root \(\rho\) of \(\Delta\) by \(-\rho\), the weighted branch coefficients satisfy
 	\[
 	c_+^{u}\equiv0,
 	\qquad
 	d_+^{u}\equiv0
 	\qquad\text{on }U.
 	\]
 	Equivalently, the first row of \(\Phi_I(u)\) annihilates the spectral projection
 	\(E_+\) associated with one local eigenvalue branch:
 	\[
 	(\Phi_I(u)E_+)_{1\bullet}\equiv(0,0).
 	\]
 	Then the complementary local eigenvalue branches \(\lambda_-\) patch together
 	to an entire function \(\lambda\in\mathcal O(\mathbb C_I)\), and, for every
 	\(f\in\mathcal{SR}(\mathbb H)\) with \(f_I=F+GJ\),
 	\begin{equation}\label{eq:weighted-one-branch-formula}
 		\begin{aligned}
 			\bigl(\Phi_I(W_{(u,\psi)}f)\bigr)_{11}
 			&=AF(\lambda)-DG^\#(\lambda),\\
 			\bigl(\Phi_I(W_{(u,\psi)}f)\bigr)_{12}
 			&=DF^\#(\lambda)+AG(\lambda).
 		\end{aligned}
 	\end{equation}
 \end{proposition}
 
 \begin{proof}
 	With the above labeling, \(c_+^{u}\equiv d_+^{u}\equiv0\). By \eqref{eq:weighted-coefficient-sums}, we have
 	$c_-^{u}=A$ and $d_-^{u}=D$, and the first row of $\Phi_I(u)$ annihilates
 	$\Psi-\lambda_-\mathrm I_2$. Thus
 	\begin{equation}\label{eq:weighted-row-eigen-identities}
 		AP-DQ^\#=\lambda_-A,
 		\qquad
 		AQ+DP^\#=\lambda_-D.
 	\end{equation}
 Since \(u\not\equiv0\), \(A\) and \(D\) do not vanish simultaneously on any
 open set; hence the two local definitions agree on a nonempty open subset of
 each overlap and therefore agree everywhere by analytic continuation.  They therefore patch to a holomorphic function on
 	$\mathcal N=\{\Delta\ne0\}$. At a zero \(z_0\) of \(\Delta\), cancel the largest common power of
 	\(z-z_0\) from \(A\) and \(D\), and keep the notation \(A,D\) for the
 	resulting holomorphic functions. Thus \(A\) and \(D\) have no common factor
 	\(z-z_0\), and in particular \(A(z_0)\) and \(D(z_0)\) are not both zero.
 	After this cancellation, the identities in
\eqref{eq:weighted-row-eigen-identities} still hold on a punctured
neighbourhood of \(z_0\). If \(A(z_0)\ne0\), then
 	\[
 	\lambda=\frac{AP-DQ^\#}{A}
 	\]
 	extends holomorphically across \(z_0\); if \(D(z_0)\ne0\), then
 	\[
 	\lambda=\frac{AQ+DP^\#}{D}
 	\]
 	does so. Then \(\lambda\) has a removable singularity at \(z_0\). Hence $\lambda$ is entire. The functional calculus then gives
 	\[
 	(\Phi_I(u)H(\Psi))_{11}=AH(\lambda),
 	\qquad
 	(\Phi_I(u)H(\Psi))_{12}=DH(\lambda),
 	\]
 	which yields \eqref{eq:weighted-one-branch-formula}.
 \end{proof}
 
 For the same-branch regime, define
 \begin{equation}\label{eq:weighted-one-branch-measure}
 	\begin{aligned}
 		\mu_{(u,\lambda,I)}(\mathcal A)
 		:={}&
 		\int_{\lambda^{-1}(\mathcal A)}
 		|A(z)|^q\mathrm e^{-\frac{\alpha q}{2}|z|^2}
 		\,\mathrm d m_{2,I}(z)\\
 		&+
 		\int_{(\lambda^\#)^{-1}(\mathcal A)}
 		|D^\#(z)|^q\mathrm e^{-\frac{\alpha q}{2}|z|^2}
 		\,\mathrm d m_{2,I}(z),
 	\end{aligned}
 \end{equation}
 and let
 $\mathrm d\nu_{(u,\lambda,I)}(\zeta)
 =\mathrm e^{\alpha q|\zeta|^2/2}\,
 \mathrm d\mu_{(u,\lambda,I)}(\zeta)$. Then
 \begin{equation}\label{eq:weighted-one-branch-Berezin}
 	\mathcal B_{(u,\psi,I)}(w)
 	\approx
 	\widetilde\nu_{(u,\lambda,I)}(w).
 \end{equation}
 Consequently, the same-branch case is again a single scalar
 Fock--Carleson embedding. In particular, boundedness forces
 $\lambda(z)=az+b$ with $|a|\le1$ whenever $u\not\equiv0$.
 
 By \eqref{eq:weighted-HL-local}, if
 $\mathcal H_u\equiv\mathcal L_u\equiv0$, then on every local region one
 coefficient in each pair vanishes. Proposition~\ref{prop:weighted-mixed-degeneracy}
 shows that, when $Q\not\equiv0$, the two vanishing signs must coincide. Hence
 Proposition~\ref{prop:weighted-same-branch-reduction} applies. The only remaining
 case is
 \begin{equation}\label{eq:weighted-genuine-two-branch}
 	\mathcal H_u\not\equiv0
 	\qquad\text{or}\qquad
 	\mathcal L_u\not\equiv0.
 \end{equation} 
If either
\[
\sup_{w\in\mathbb C_I}\mathcal B_{(u,\psi,I)}(w)<\infty,
\]
or \(0<q<p<\infty\) and
\[
\mathcal B_{(u,\psi,I)}\in L^{p/(p-q)}(\mathbb C_I),
\]
then all conclusions of Theorem~\ref{thm:composition-spectral-rigidity}
hold for the same functions \(\Delta,\tau,\lambda_\pm\). 
\paragraph{The  two-branch regime.}
 \label{subsec:weighted-transfer}
 
 Assume \eqref{eq:weighted-genuine-two-branch}. Use the same good/singular
 regions $U_0,U_1,E,\Gamma$ as in
 Lemma~\ref{lem:composition-good-singular-decomposition}, with the same
 constant-discriminant convention. On each nonempty $U_j$, let
 $c_{\pm,j}^{u},d_{\pm,j}^{u}$ be the coefficients
 \eqref{eq:weighted-branch-coefficients}. Define the pull-back measures by
 \begin{align}
 	\nu_{j,u}^{c,\pm}(\mathcal A)
 	&:=
 	\int_{U_j\cap\lambda_{\pm,j}^{-1}(\mathcal A)}
 	|c_{\pm,j}^{u}(z)|^q
 	\mathrm e^{-\frac{\alpha q}{2}|z|^2}
 	\,\mathrm d m_{2,I}(z),
 	\nonumber\\
 	\nu_{j,u}^{d,\pm}(\mathcal A)
 	&:=
 	\int_{U_j\cap\lambda_{\pm,j}^{-1}(\mathcal A)}
 	|d_{\pm,j}^{u}(z)|^q
 	\mathrm e^{-\frac{\alpha q}{2}|z|^2}
 	\,\mathrm d m_{2,I}(z),\nonumber
 \end{align}
 and
 \begin{equation}\label{eq:weighted-Lambda}
 	\mathrm d\Lambda_{j,u}^{c,\pm}(\zeta)
 	:=\mathrm e^{\frac{\alpha q}{2}|\zeta|^2}
 	\,\mathrm d\nu_{j,u}^{c,\pm}(\zeta),
 	\qquad
 	\mathrm d\Lambda_{j,u}^{d,\pm}(\zeta)
 	:=\mathrm e^{\frac{\alpha q}{2}|\zeta|^2}
 	\,\mathrm d\nu_{j,u}^{d,\pm}(\zeta).
 	\nonumber
 \end{equation}
 Also set
 \begin{align}
 	\mathcal B_{U_j}^{u}(w)
 	&:=
 	\int_{U_j}
 	\left|\bigl(u\star(k_w\bullet\psi)\bigr)_I(z)\right|^q
 	\mathrm e^{-\frac{\alpha q}{2}|z|^2}
 	\,\mathrm d m_{2,I}(z),\nonumber
 	\\
 	\mathcal B_E^{u}(w)
 	&:=
 	\int_E
 	\left|\bigl(u\star(k_w\bullet\psi)\bigr)_I(z)\right|^q
 	\mathrm e^{-\frac{\alpha q}{2}|z|^2}
 	\,\mathrm d m_{2,I}(z).\nonumber
 \end{align}
 Then
 \begin{equation}\label{eq:weighted-B-decomposition}
 	\mathcal B_{(u,\psi,I)}
 	=
 	\mathcal B_{U_0}^{u}+\mathcal B_{U_1}^{u}+\mathcal B_E^{u}.
 	\nonumber
 \end{equation}
 
 \begin{proposition} 
 	\label{prop:weighted-two-branch-transfer}
Assume \(u\not\equiv0\), \(Q\not\equiv0\), and
\eqref{eq:weighted-genuine-two-branch}. Then the following assertions hold.
 	
 	\smallskip
 	\noindent\textup{(i)} If either
 	\[
 	\sup_w\mathcal B_{(u,\psi,I)}(w)<\infty,
 	\]
 	or, for $0<q<p<\infty$,
 	\[
 	\mathcal B_{(u,\psi,I)}\in L^{p/(p-q)}(\mathbb C_I),
 	\]
 	then all conclusions of
 	Theorem~\ref{thm:composition-spectral-rigidity} hold for the same functions
 	$\Delta$, $\tau$, and $\lambda_\pm$.
 	
 	\smallskip
 	\noindent\textup{(ii)} If $0<p\le q<\infty$ and
 	$\sup_w\mathcal B_{(u,\psi,I)}(w)<\infty$, then the measures
 	\[
 	\Lambda_{j,u}^{c,\pm},
 	\qquad
 	\Lambda_{j,u}^{d,\pm},
 	\qquad j=0,1,
 	\]
 	are $(p,q)$-Fock--Carleson measures. If
 	$\mathcal B_{(u,\psi,I)}(w)\to0$, they are vanishing
 	$(p,q)$-Fock--Carleson measures.
 	
 	\smallskip
 	\noindent\textup{(iii)} Let $0<q<p<\infty$ and $s=p/(p-q)$. If
 	$\mathcal B_{(u,\psi,I)}\in L^s$, then all the preceding good-region measures
 	are vanishing $(p,q)$-Fock--Carleson measures. Conversely, if
 	$W_{(u,\psi)}$ is bounded, then
 	\[
 	\mathcal B_{U_j}^{u}\in L^s,
 	\qquad
 	\|\mathcal B_{U_j}^{u}\|_{L^s}^{1/q}
 	\lesssim
 	\|W_{(u,\psi)}\|,
 	\qquad j=0,1.
 	\]
 	
 	\smallskip
 	\noindent\textup{(iv)} If $\Delta$ is nonconstant, the singular neighborhood may
 	be chosen so that, for every $\varepsilon>0$,
 	\[
 	\int_E
 	\left|\bigl(W_{(u,\psi)}f\bigr)_I(z)\right|^q
 	\mathrm e^{-\frac{\alpha q}{2}|z|^2}
 	\,\mathrm d m_{2,I}(z)
 	\le
 	\varepsilon\|f\|_{F_\alpha^p(\mathbb H)}^q.
 	\]
 	Moreover, when $0<q<p<\infty$,
 	$\mathcal B_E^{u}\in L^s$, and
 	\[
 	\int_E
 	\left|\bigl(W_{(u,\psi)}f\bigr)_I(z)\right|^q
 	\mathrm e^{-\frac{\alpha q}{2}|z|^2}
 	\,\mathrm d m_{2,I}(z)
 	\lesssim_E
 	\|\mathcal B_E^{u}\|_{L^s}
 	\|f\|_{F_\alpha^p(\mathbb H)}^q.
 	\]
 	If $W_{(u,\psi)}$ is bounded, then
 	\[
 	\|\mathcal B_E^{u}\|_{L^s}^{1/q}
 	\lesssim_E
 	\|W_{(u,\psi)}\|.
 	\]
 \end{proposition}
 \begin{proof}
 	We record only the changes needed to pass from the unweighted case to the weighted case.  Putting \(F=k_w\) and \(G=0\) in
\eqref{eq:weighted-local-formula-11}, we obtain, on every nondegenerate
 	local region,
 	\[
 	\bigl(\Phi_I(W_{(u,\psi)}k_w)\bigr)_{11}(z)
 	=
 	\sum_{\sigma=\pm}
 	c_\sigma^u(z)k_w(\lambda_\sigma(z)).
 	\]
 	Applying the same local submean estimate as in the proof of
 	Theorem~\ref{thm:composition-spectral-rigidity}, and using
 	\eqref{eq:weighted-testing-quantity}, we obtain
 	\begin{equation}\label{eq:weighted-local-kernel-test-c}
 		\left|
 		\sum_{\sigma=\pm}
 		\overline{c_\sigma^u(z)}K_{\lambda_\sigma(z)}(w)
 		\right|^q
 		\mathrm e^{-\frac{\alpha q}{2}|w|^2}
 		\lesssim
 		\mathrm e^{\frac{\alpha q}{2}|z|^2}
 		\mathcal B_{(u,\psi,I)}(w).
 	\end{equation}
 	The \((1,2)\)-entry gives  the
 	corresponding estimate with \(d_\sigma^u\) in place of \(c_\sigma^u\).
 	
 Thus the separated coefficient estimates from the unweighted case remain valid
under the replacement rule \eqref{eq:weighted-global-substitution}. 
 	If \(\mathcal H_u\not\equiv0\), we use \eqref{eq:weighted-HL-local}; 
 	if \(\mathcal H_u\equiv0\), then \(\mathcal L_u\not\equiv0\) by
 	\eqref{eq:weighted-genuine-two-branch}, and we use \eqref{eq:weighted-HL-local}. 
 	The Jensen and growth argument in
 	Theorem~\ref{thm:composition-spectral-rigidity} then applies verbatim, with
 	\(\mathcal H_u\) or \(\mathcal L_u\) in place of 	\(\mathcal H \) or \(\mathcal L \) . This proves \textup{(i)}.
 	
 	On the good regions, the proofs of
 	Propositions~\ref{prop:composition-good-p-le-q} and
 	\ref{prop:composition-good-q-less-p} use only the separated coefficient
 	estimates, the affine asymptotics of the branches \(\lambda_\pm\), and the
 	local two-branch formulas. These ingredients have just been established in the weighted form. Thus the same arguments yield \textup{(ii)} and
 	\textup{(iii)}. If one of the weighted coefficients vanishes identically, the
 	corresponding measure is zero and there is nothing to prove for that branch.
 	
 	It remains to indicate the estimates on the singular region. Since \(E\) is compact, by \eqref{eq:composition-compact-set-estimate} and the boundedness of
\(\Phi_I(u)\) on \(E\),  
 	\[
 	\sup_{z\in E}
 	\|\Phi_I(W_{(u,\psi)}f)(z)\|
 	\lesssim 
 	\|f\|_{F_\alpha^p(\mathbb H)} .
 	\]
 	Hence, if \(\Delta\) is nonconstant, the disks forming \(E\) may be chosen with
 	arbitrarily small total area, and this gives the first estimate in
 	\textup{(iv)}.
 	
 	For the \(L^s\)-estimate on \(E\), by \eqref{eq:composition-kernel-matrix-exponential}, 
 	\[
 	\sup_{z\in E}
 	\|\Phi_I(u)(z)\Phi_I(k_w\bullet\psi)(z)\|
 	\lesssim
 	\mathrm e^{-\frac{\alpha}{2}|w|^2+C_E|w|}.
 	\]
 	Therefore
 	\[
 	\mathcal B_E^u(w)
 	\lesssim
 	\mathcal B_E^u(0)\,
 	\mathrm e^{-\frac{\alpha q}{2}|w|^2+C_E|w|}
 	\lesssim
 	\mathcal B_E^u(0)\mathrm e^{-c|w|^2}
 	\]
 	for some \(c>0\). Thus
 	\[
 	\mathcal B_E^u\in L^s(\mathbb C_I),
 	\qquad
 	\|\mathcal B_E^u\|_{L^s}
 	\lesssim_E
 	\mathcal B_E^u(0).
 	\]
 	Moreover,
 	\[
 	\Phi_I(u)\Phi_I(k_w\bullet\psi)\longrightarrow \Phi_I(u)
 	\]
 	uniformly on \(E\) as \(w\to0\). Hence
 	\[
 	\mathcal B_E^u(w)\longrightarrow \mathcal B_E^u(0),
 	\qquad w\to0,
 	\]
 	and consequently
 	\[
 	\mathcal B_E^u(0)
 	\lesssim
 	\|\mathcal B_E^u\|_{L^s}.
 	\]
 	
 	Finally, for \(f\in F_\alpha^p(\mathbb H)\), the compact-set estimate gives
 	\[
 	\begin{aligned}
 		&\int_E
 		\left|\bigl(W_{(u,\psi)}f\bigr)_I(z)\right|^q
 		\mathrm e^{-\frac{\alpha q}{2}|z|^2}
 		\,\mathrm d m_{2,I}(z)   \lesssim
 		\|f\|_{F_\alpha^p(\mathbb H)}^q
 		\int_E
 		\|\Phi_I(u)(z)\|^q
 		\mathrm e^{-\frac{\alpha q}{2}|z|^2}
 		\,\mathrm d m_{2,I}(z).
 	\end{aligned}
 	\]
 	Since \(E=E^\#\), 
 	the last integral is controlled by
 	\[
 	\int_E |u_I(z)|^q
 	\mathrm e^{-\frac{\alpha q}{2}|z|^2}
 	\,\mathrm d m_{2,I}(z)
 	=
 	\mathcal B_E^u(0).
 	\]
 	Combining this with the preceding estimate of
 	\(\mathcal B_E^u(0)\) by \(\|\mathcal B_E^u\|_{L^s}\) proves the singular
 	embedding estimate in \textup{(iv)}.
 	
 	If \(W_{(u,\psi)}\) is bounded, then
 	\[
 	\mathcal B_E^u(0)
 	\le
 	\mathcal B_{(u,\psi,I)}(0)
 	\approx
 	\|W_{(u,\psi)}k_0\|_{F_\alpha^q(\mathbb H)}^q
 	\lesssim
 	\|W_{(u,\psi)}\|^q.
 	\]
 	Together with
 	\(\|\mathcal B_E^u\|_{L^s}\lesssim_E \mathcal B_E^u(0)\), this gives
 	\[
 	\|\mathcal B_E^u\|_{L^s}^{1/q}
 	\lesssim_E
 	\|W_{(u,\psi)}\|.
 	\]
 	This completes the proof.
 \end{proof}
 
\paragraph{Proof of the weighted criteria.}
 \label{subsec:weighted-main-proofs}
 
 \begin{proof}[Proof of Theorem~\ref{thm:weighted-p-le-q}]
 	If $u\equiv0$, the result is immediate. Assume $u\not\equiv0$. Since
 	$\|k_w\|_{F_\alpha^p(\mathbb H)}\approx1$,
 	\begin{equation}\label{eq:weighted-kernel-test}
 		\mathcal B_{(u,\psi,I)}(w)
 		\approx
 		\|W_{(u,\psi)}k_w\|_{F_\alpha^q(\mathbb H)}^q.
 	\end{equation}
 	Thus boundedness implies the testing condition and the lower norm estimate.
 	
 	Conversely, assume the testing quantity is bounded. If $Q\equiv0$, use the
 	scalar reduction \eqref{eq:weighted-scalar-Berezin}. If
 	$Q\not\equiv0$ and $\mathcal H_u\equiv\mathcal L_u\equiv0$, use the
 	same-branch reduction \eqref{eq:weighted-one-branch-Berezin}. In the remaining
 	two-branch case, apply
 	Proposition~\ref{prop:weighted-two-branch-transfer}\textup{(ii),(iv)} and the
 	local formula \eqref{eq:weighted-local-formula-11}--
 	\eqref{eq:weighted-local-formula-12}. These three cases give
 	\[
 	\|W_{(u,\psi)}f\|_{F_\alpha^q(\mathbb H)}
 	\lesssim
 	\sup_w\mathcal B_{(u,\psi,I)}(w)^{1/q}
 	\|f\|_{F_\alpha^p(\mathbb H)}.
 	\]
 	Together with \eqref{eq:weighted-kernel-test}, this proves the norm equivalence.
 	
 	For compactness, necessity again follows from
 	\eqref{eq:weighted-kernel-test}. Conversely, in the scalar and same-branch regimes, the associated scalar pull-back measure is vanishing. In the
 	two-branch regime, the good-region measures are vanishing by
 	Proposition~\ref{prop:weighted-two-branch-transfer}\textup{(ii)}. By \eqref{eq:composition-compact-open-continuity},
\(\Phi_I(f_n\bullet\psi)\to0\) uniformly on compact subsets; multiplication by
\(\Phi_I(u)\) preserves this convergence on compact sets. Hence the
 	singular contribution tends to zero. Therefore $W_{(u,\psi)}$ is compact.
 	
 	A condition on one slice yields boundedness or compactness of the operator; the
 	kernel test then yields the corresponding condition on every slice.
 \end{proof}
 
 \begin{proof}[Proof of Theorem~\ref{thm:weighted-q-less-p}]
 	Put $s=p/(p-q)$. If $u\equiv0$, there is nothing to prove.
 	
 	Assume first that
 	$\mathcal B_{(u,\psi,I)}\in L^s$. In the scalar regime use
 	\eqref{eq:weighted-scalar-Berezin}; in the same-branch regime use
 	\eqref{eq:weighted-one-branch-Berezin}. In both cases the complex theorem for
 	$q<p$ gives a vanishing Carleson embedding. In the two-branch regime,
 	Proposition~\ref{prop:weighted-two-branch-transfer}\textup{(iii),(iv)} gives
 	vanishing embeddings on the good regions and the required singular estimate.
 	Thus $W_{(u,\psi)}$ is compact, and
 	\[
 	\|W_{(u,\psi)}\|
 	\lesssim
 	\|\mathcal B_{(u,\psi,I)}\|_{L^s}^{1/q}.
 	\]
 	
 	Conversely, suppose that $W_{(u,\psi)}$ is bounded. The scalar and same-branch
 	reductions give the $L^s$ condition through the corresponding complex
 	Fock--Carleson theorem. In the two-branch regime,
 	Proposition~\ref{prop:weighted-two-branch-transfer}\textup{(iii),(iv)} gives
 	\[
 	\mathcal B_{U_0}^{u},\ \mathcal B_{U_1}^{u},\ \mathcal B_E^{u}
 	\in L^s
 	\]
 	and
 	\[
 	\|\mathcal B_{(u,\psi,I)}\|_{L^s}^{1/q}
 	\lesssim
 	\|W_{(u,\psi)}\|.
 	\]
 	Together with the preceding upper estimate, this proves the norm equivalence.
 	Since compactness implies boundedness, all three assertions are equivalent.
 	As before, validity on one slice implies validity on every slice.
 \end{proof}

  \subsection{Products of Volterra-Type Integral and Composition Operators}
  \label{sec:volterra-composition-products}
  
 In this section we study the operator \(V_{(g,\psi)}\) introduced in
Definition~\ref{def:volterra-composition}. 
  The results follow from the weighted composition theory of
  Section~\ref{sec:weighted-composition-operators} by means of the
  following slice Littlewood--Paley estimate.
 \begin{lemma}
 	\label{lem:little}
 	Let \(0<p<\infty\) and \(h\in\mathcal{SR}(\mathbb H)\). Then, for every
 	\(I\in\mathbb S\),
 	\[
 	\int_{\mathbb C_I}
 	|h(z)|^p
 	\mathrm e^{-\frac{\alpha p}{2}|z|^2}
 	\,\mathrm d m_{2,I}(z)
 	\approx
 	|h(0)|^p
 	+
 	\int_{\mathbb C_I}
 	\frac{|h'(z)|^p}{(1+|z|)^p}
 	\mathrm e^{-\frac{\alpha p}{2}|z|^2}
 	\,\mathrm d m_{2,I}(z).
 	\]
 	The constants are independent of \(h\) and of the chosen slice.
 \end{lemma}
 
 \begin{proof}
 	Choose \(J\in\mathbb S\) with \(J\perp I\) and write $
 	h_I=H+KJ$. 
 	Then
 	\[
 	|h_I(z)|^p\approx |H(z)|^p+|K(z)|^p,
 	\qquad
 	|(h')_I(z)|^p\approx |H'(z)|^p+|K'(z)|^p.
 	\]
 	Applying the classical Littlewood--Paley equivalence on complex Fock
 	spaces separately to \(H\) and \(K\), see
 	\cite[Proposition~1]{MR2957216} and
 	\cite[Lemma~1]{MR3192295}, gives the result.
 \end{proof}

For \(I\in\mathbb S\) and \(w\in\mathbb C_I\), define
\begin{equation}\label{eq:volterra-testing-quantity}
	\mathcal B_{(g,\psi,I)}^V(w)
	:=
	\int_{\mathbb C_I}
	\left|
	\bigl(g'\star(k_w\bullet\psi)\bigr)_I(z)
	\right|^q
	\frac{\mathrm e^{-\frac{\alpha q}{2}|z|^2}}
	{(1+|z|)^q}
	\,\mathrm d m_{2,I}(z).
\end{equation}

The reduction to the weighted-composition case is as follows. Since
\[
\bigl(V_{(g,\psi)}f\bigr)(0)=0,
\qquad
\bigl(V_{(g,\psi)}f\bigr)'
=
g'\star(f\bullet\psi),
\]
Lemma~\ref{lem:little} gives
\begin{equation}\label{eq:volterra-LP-reduction}
	\|V_{(g,\psi)}f\|_{F_\alpha^q(\mathbb H)}^q
	\approx
	\int_{\mathbb C_I}
	\left|
	\bigl(g'\star(f\bullet\psi)\bigr)_I(z)
	\right|^q
	\frac{\mathrm e^{-\frac{\alpha q}{2}|z|^2}}
	{(1+|z|)^q}
	\,\mathrm d m_{2,I}(z).
\end{equation}
Thus the proofs of the weighted-composition criteria apply with
\(u\) replaced by \(g'\), and with the measure
\[
\mathrm e^{-\frac{\alpha q}{2}|z|^2}\,\mathrm d m_{2,I}(z)
\]
replaced by
\[
\frac{\mathrm e^{-\frac{\alpha q}{2}|z|^2}}
{(1+|z|)^q}
\,\mathrm d m_{2,I}(z).
\]
Equivalently, on each nondegenerate local region, the coefficients
\(c_\pm^u,d_\pm^u\) in \eqref{eq:weighted-branch-coefficients} are replaced by
\[
c_\pm^V:=(\Phi_I(g')E_\pm)_{11},
\qquad
d_\pm^V:=(\Phi_I(g')E_\pm)_{12}.
\]
The formula \eqref{eq:composition-M-def} becomes
\[
M_z^V(H,\lambda)
:=
\frac{|H(z)|}{1+|z|}
\exp\!\left\{
\frac{\alpha}{2}
\bigl(|\lambda(z)|^2-|z|^2\bigr)
\right\}.
\]
Since \(1+|\zeta|\asymp1+|z|\) on every fixed disk \(D_I(z,r)\), this extra
factor does not affect the local submean estimates. On the compact singular region
it is bounded above and below, while on the good unbounded regions it only improves
the estimates. Therefore no new argument is needed, and we obtain the following
criteria.

\begin{theorem}\label{thm:volterra-p-le-q}
	Let \(g,\psi\in\mathcal{SR}(\mathbb H)\) and \(0<p\le q<\infty\). Then the
	following assertions hold.
	
	\smallskip
	\noindent\textup{(i)}
	The operator $
	V_{(g,\psi)}:
	F_\alpha^p(\mathbb H)\longrightarrow F_\alpha^q(\mathbb H)$ 
	is bounded if and only if, for some, equivalently for every,
	\(I\in\mathbb S\),
	\[
	\sup_{w\in\mathbb C_I}
	\mathcal B_{(g,\psi,I)}^V(w)<\infty .
	\]
	Moreover,
	\[
	\|V_{(g,\psi)}\|
	\approx
	\sup_{w\in\mathbb C_I}
	\mathcal B_{(g,\psi,I)}^V(w)^{1/q}.
	\]
	
	\smallskip
	\noindent\textup{(ii)}
	The operator \(V_{(g,\psi)}\) is compact if and only if, for some,
	equivalently for every, \(I\in\mathbb S\),
	\[
	\mathcal B_{(g,\psi,I)}^V(w)\longrightarrow0
	\qquad
	\text{as } |w|\to\infty,\quad w\in\mathbb C_I .
	\]
\end{theorem}

\begin{theorem}\label{thm:volterra-q-less-p}
	Let \(g,\psi\in\mathcal{SR}(\mathbb H)\) and \(0<q<p<\infty\). Put
	\(s=p/(p-q)\). Then the following assertions are equivalent:
	\begin{enumerate}[(i)]
		\item
		\(V_{(g,\psi)}:F_\alpha^p(\mathbb H)\to F_\alpha^q(\mathbb H)\) is bounded;
		
		\item
		\(V_{(g,\psi)}:F_\alpha^p(\mathbb H)\to F_\alpha^q(\mathbb H)\) is compact;
		
		\item
		for some, equivalently for every, \(I\in\mathbb S\),
		\[
		\mathcal B_{(g,\psi,I)}^V
		\in L^s(\mathbb C_I,\mathrm d m_{2,I}).
		\]
	\end{enumerate}
	Moreover,
	\[
	\|V_{(g,\psi)}\|
	\approx
	\left(
	\int_{\mathbb C_I}
	\mathcal B_{(g,\psi,I)}^V(w)^s
	\,\mathrm d m_{2,I}(w)
	\right)^{1/(qs)}.
	\]
\end{theorem}

	
	\appendix
	 
	\section*{Acknowledgments}
	The authors would like to express their sincere gratitude to Professor Yuxia Liang
	of Tianjin Normal University for her constant encouragement, valuable discussions,
	and helpful comments on this paper.

	

	\bibliographystyle{amsplain} 
	\bibliography{references} 

@article {MR3482788,
	AUTHOR = {Ren, Guangbin and Wang, Xieping},
	TITLE = {Slice regular composition operators},
	JOURNAL = {Complex Var. Elliptic Equ.},
	FJOURNAL = {Complex Variables and Elliptic Equations. An International
	Journal},
	VOLUME = {61},
	YEAR = {2016},
	NUMBER = {5},
	PAGES = {682--711},
	ISSN = {1747-6933,1747-6941},
	MRCLASS = {30G35 (32A26)},
	MRNUMBER = {3482788},
	DOI = {10.1080/17476933.2015.1113270},
	URL = {https://doi.org/10.1080/17476933.2015.1113270},
}

@article{ueki2007weighted,
	title={Weighted composition operator on the Fock space},
	author={Ueki, Sei-Ichiro},
	journal={Proceedings of the American Mathematical Society},
	volume={135},
	number={5},
	pages={1405--1410},
	year={2007}
}

@article {MR881273,
	AUTHOR = {Shapiro, Joel H.},
	TITLE = {The essential norm of a composition operator},
	JOURNAL = {Ann. of Math. (2)},
	FJOURNAL = {Annals of Mathematics. Second Series},
	VOLUME = {125},
	YEAR = {1987},
	NUMBER = {2},
	PAGES = {375--404},
	ISSN = {0003-486X,1939-8980},
	MRCLASS = {47B38},
	MRNUMBER = {881273},
	MRREVIEWER = {Carl\ C.\ Cowen},
	DOI = {10.2307/1971314},
	URL = {https://doi.org/10.2307/1971314},
}

@incollection {MR2198381,
	AUTHOR = {Smith, Wayne},
	TITLE = {Brennan's conjecture for weighted composition operators},
	BOOKTITLE = {Recent advances in operator-related function theory},
	SERIES = {Contemp. Math.},
	VOLUME = {393},
	PAGES = {209--214},
	PUBLISHER = {Amer. Math. Soc., Providence, RI},
	YEAR = {2006},
	ISBN = {0-8218-3925-X},
	MRCLASS = {47B33 (30C35 30D55)},
	MRNUMBER = {2198381},
	MRREVIEWER = {Ruhan\ Zhao},
	DOI = {10.1090/conm/393/07380},
	URL = {https://doi.org/10.1090/conm/393/07380},
}

@article {MR2609242,
    AUTHOR = {Isralowitz, Josh and Zhu, Kehe},
     TITLE = {Toeplitz operators on the {F}ock space},
   JOURNAL = {Integral Equations Operator Theory},
  FJOURNAL = {Integral Equations and Operator Theory},
    VOLUME = {66},
      YEAR = {2010},
    NUMBER = {4},
     PAGES = {593--611},
      ISSN = {0378-620X,1420-8989},
   MRCLASS = {47B35 (30H20 46E22 47B07)},
  MRNUMBER = {2609242},
MRREVIEWER = {Georg\ T.\ Schneider},
       DOI = {10.1007/s00020-010-1768-9},
       URL = {https://doi.org/10.1007/s00020-010-1768-9},
}

@article {MR3248473,
    AUTHOR = {Hu, Zhangjian and Lv, Xiaofen},
     TITLE = {Toeplitz operators on {F}ock spaces {$F^p(\varphi)$}},
   JOURNAL = {Integral Equations Operator Theory},
  FJOURNAL = {Integral Equations and Operator Theory},
    VOLUME = {80},
      YEAR = {2014},
    NUMBER = {1},
     PAGES = {33--59},
      ISSN = {0378-620X,1420-8989},
   MRCLASS = {47B38 (32A37)},
  MRNUMBER = {3248473},
MRREVIEWER = {N.\ L.\ Vasilevski\u i},
       DOI = {10.1007/s00020-014-2168-3},
       URL = {https://doi.org/10.1007/s00020-014-2168-3},
}

@article {MR2964691,
    AUTHOR = {Cho, Hong Rae and Zhu, Kehe},
     TITLE = {Fock-{S}obolev spaces and their {C}arleson measures},
   JOURNAL = {J. Funct. Anal.},
  FJOURNAL = {Journal of Functional Analysis},
    VOLUME = {263},
      YEAR = {2012},
    NUMBER = {8},
     PAGES = {2483--2506},
      ISSN = {0022-1236,1096-0783},
   MRCLASS = {46E35 (28A12 32A37 42B35)},
  MRNUMBER = {2964691},
MRREVIEWER = {Agnieszka\ Ka\l amajska},
       DOI = {10.1016/j.jfa.2012.08.003},
       URL = {https://doi.org/10.1016/j.jfa.2012.08.003},
}

@article {MR2342670,
    author   = {{\v{Z}}eljko {\v{C}}u{\v{c}}kovi{\'c} and Ruhan Zhao},
     TITLE = {Weighted composition operators between different weighted
              {B}ergman spaces and different {H}ardy spaces},
   JOURNAL = {Illinois J. Math.},
  FJOURNAL = {Illinois Journal of Mathematics},
    VOLUME = {51},
      YEAR = {2007},
    NUMBER = {2},
     PAGES = {479--498},
      ISSN = {0019-2082,1945-6581},
   MRCLASS = {47B33 (30D55 46E15)},
  MRNUMBER = {2342670},
MRREVIEWER = {William\ Thomas\ Ross},
       URL = {http://projecteuclid.org/euclid.ijm/1258138425},
}

@article{MR2078907,
  author   = {{\v{Z}}eljko {\v{C}}u{\v{c}}kovi{\'c} and Ruhan Zhao},
  title    = {Weighted composition operators on the {B}ergman space},
  journal  = {J. London Math. Soc. (2)},
  fjournal = {Journal of the London Mathematical Society. Second Series},
  volume   = {70},
  year     = {2004},
  number   = {2},
  pages    = {499--511},
  issn     = {0024-6107,1469-7750},
  mrclass  = {47B33 (30D55 47B10)},
  mrnumber = {2078907},
  doi      = {10.1112/S0024610704005605},
  url      = {https://doi.org/10.1112/S0024610704005605},
}

@article {MR2886614,
	AUTHOR = {Matache, Valentin and Smith, Wayne},
	TITLE = {Composition operators on a class of analytic function spaces
	related to {B}rennan's conjecture},
	JOURNAL = {Complex Anal. Oper. Theory},
	FJOURNAL = {Complex Analysis and Operator Theory},
	VOLUME = {6},
	YEAR = {2012},
	NUMBER = {1},
	PAGES = {139--162},
	ISSN = {1661-8254,1661-8262},
	MRCLASS = {47B33 (30H20)},
	MRNUMBER = {2886614},
	MRREVIEWER = {Julio\ C\'esar\ Ramos Fern\'andez},
	DOI = {10.1007/s11785-010-0090-5},
	URL = {https://doi.org/10.1007/s11785-010-0090-5},
}

@article {MR854144,
	AUTHOR = {MacCluer, Barbara D. and Shapiro, Joel H.},
	TITLE = {Angular derivatives and compact composition operators on the
	{H}ardy and {B}ergman spaces},
	JOURNAL = {Canad. J. Math.},
	FJOURNAL = {Canadian Journal of Mathematics. Journal Canadien de
	Math\'ematiques},
	VOLUME = {38},
	YEAR = {1986},
	NUMBER = {4},
	PAGES = {878--906},
	ISSN = {0008-414X,1496-4279},
	MRCLASS = {47B05 (30D55 47B38)},
	MRNUMBER = {854144},
	MRREVIEWER = {Warren\ R.\ Wogen},
	DOI = {10.4153/CJM-1986-043-4},
	URL = {https://doi.org/10.4153/CJM-1986-043-4},
}

@article {MR223914,
	AUTHOR = {Nordgren, Eric A.},
	TITLE = {Composition operators},
	JOURNAL = {Canadian J. Math.},
	FJOURNAL = {Canadian Journal of Mathematics. Journal Canadien de
	Math\'ematiques},
	VOLUME = {20},
	YEAR = {1968},
	PAGES = {442--449},
	ISSN = {0008-414X,1496-4279},
	MRCLASS = {47.25},
	MRNUMBER = {223914},
	MRREVIEWER = {E.\ R.\ Deal},
	DOI = {10.4153/CJM-1968-040-4},
	URL = {https://doi.org/10.4153/CJM-1968-040-4},
}

@article {MR192062,
	AUTHOR = {Ryff, John V.},
	TITLE = {Subordinate {$H\sp{p}$} functions},
	JOURNAL = {Duke Math. J.},
	FJOURNAL = {Duke Mathematical Journal},
	VOLUME = {33},
	YEAR = {1966},
	PAGES = {347--354},
	ISSN = {0012-7094,1547-7398},
	MRCLASS = {30.67},
	MRNUMBER = {192062},
	MRREVIEWER = {R.\ J.\ Libera},
	URL = {http://projecteuclid.org/euclid.dmj/1077376389},
}

@article {MR1575208,
	AUTHOR = {Littlewood, J. E.},
	TITLE = {On {I}nequalities in the {T}heory of {F}unctions},
	JOURNAL = {Proc. London Math. Soc. (2)},
	FJOURNAL = {Proceedings of the London Mathematical Society. Second Series},
	VOLUME = {23},
	YEAR = {1925},
	NUMBER = {7},
	PAGES = {481--519},
	ISSN = {0024-6115},
	MRCLASS = {99-04},
	MRNUMBER = {1575208},
	DOI = {10.1112/plms/s2-23.1.481},
	URL = {https://doi.org/10.1112/plms/s2-23.1.481},
}

@article {MR3905527,
	AUTHOR = {Tien, Pham Trong and Khoi, Le Hai},
	TITLE = {Weighted composition operators between different {F}ock
	spaces},
	JOURNAL = {Potential Anal.},
	FJOURNAL = {Potential Analysis. An International Journal Devoted to the
	Interactions between Potential Theory, Probability Theory,
	Geometry and Functional Analysis},
	VOLUME = {50},
	YEAR = {2019},
	NUMBER = {2},
	PAGES = {171--195},
	ISSN = {0926-2601,1572-929X},
	MRCLASS = {30H20 (47B33)},
	MRNUMBER = {3905527},
	MRREVIEWER = {Maria\ Tjani},
	DOI = {10.1007/s11118-017-9678-y},
	URL = {https://doi.org/10.1007/s11118-017-9678-y},
}

@article {MR4811250,
	AUTHOR = {Yang, Zicong},
	TITLE = {The difference of weighted composition operators on {F}ock
	spaces},
	JOURNAL = {Monatsh. Math.},
	FJOURNAL = {Monatshefte f\"ur Mathematik},
	VOLUME = {205},
	YEAR = {2024},
	NUMBER = {3},
	PAGES = {667--680},
	ISSN = {0026-9255,1436-5081},
	MRCLASS = {47B33 (30H20 46E15)},
	MRNUMBER = {4811250},
	MRREVIEWER = {Trieu\ Le},
	DOI = {10.1007/s00605-024-02003-8},
	URL = {https://doi.org/10.1007/s00605-024-02003-8},
}

@article {MR2034214,
	AUTHOR = {Carswell, Brent J. and MacCluer, Barbara D. and Schuster,
	Alex},
	TITLE = {Composition operators on the {F}ock space},
	JOURNAL = {Acta Sci. Math. (Szeged)},
	FJOURNAL = {Acta Universitatis Szegediensis. Acta Scientiarum
	Mathematicarum},
	VOLUME = {69},
	YEAR = {2003},
	NUMBER = {3-4},
	PAGES = {871--887},
	ISSN = {0001-6969,2064-8316},
	MRCLASS = {47B33 (32A37)},
	MRNUMBER = {2034214},
	MRREVIEWER = {Miroslav\ Engli\v s},
}

@article {MR4384577,
	AUTHOR = {Han, Kaikai and Wang, Maofa},
	TITLE = {Slice regular weighted composition operators},
	JOURNAL = {Complex Var. Elliptic Equ.},
	FJOURNAL = {Complex Variables and Elliptic Equations. An International
	Journal},
	VOLUME = {67},
	YEAR = {2022},
	NUMBER = {1},
	PAGES = {162--223},
	ISSN = {1747-6933,1747-6941},
	MRCLASS = {30G35 (32A36 47B33)},
	MRNUMBER = {4384577},
	MRREVIEWER = {M.\ Elena\ Luna-Elizarrar\'as},
	DOI = {10.1080/17476933.2020.1818731},
	URL = {https://doi.org/10.1080/17476933.2020.1818731},
}

@book {MR1397026,
	AUTHOR = {Cowen, Carl C. and MacCluer, Barbara D.},
	TITLE = {Composition operators on spaces of analytic functions},
	SERIES = {Studies in Advanced Mathematics},
	PUBLISHER = {CRC Press, Boca Raton, FL},
	YEAR = {1995},
	PAGES = {xii+388},
	ISBN = {0-8493-8492-3},
	MRCLASS = {47B38 (30D55 46E15)},
	MRNUMBER = {1397026},
	MRREVIEWER = {John\ N.\ McDonald},
}

@book {MR1237406,
	AUTHOR = {Shapiro, Joel H.},
	TITLE = {Composition operators and classical function theory},
	SERIES = {Universitext: Tracts in Mathematics},
	PUBLISHER = {Springer-Verlag, New York},
	YEAR = {1993},
	PAGES = {xvi+223},
	ISBN = {0-387-94067-7},
	MRCLASS = {47B38 (30C99 46E20 46J15)},
	MRNUMBER = {1237406},
	MRREVIEWER = {Aristomenis\ Siskakis},
	DOI = {10.1007/978-1-4612-0887-7},
	URL = {https://doi.org/10.1007/978-1-4612-0887-7},
}

@article {MR4544164,
	AUTHOR = {Liu, Meicheng and Liang, Yuxia and Lian, Pan},
	TITLE = {Self-adjoint, unitary, and isometric weighted composition
	operators on quaternionic {F}ock space},
	JOURNAL = {Banach J. Math. Anal.},
	FJOURNAL = {Banach Journal of Mathematical Analysis},
	VOLUME = {17},
	YEAR = {2023},
	NUMBER = {2},
	PAGES = {Paper No. 24, 25},
	ISSN = {2662-2033,1735-8787},
	MRCLASS = {47B33 (30G35 30H20 47B38)},
	MRNUMBER = {4544164},
	MRREVIEWER = {Hai-Chou\ Li},
	DOI = {10.1007/s43037-023-00252-7},
	URL = {https://doi.org/10.1007/s43037-023-00252-7},
}

@article {MR3895397,
	AUTHOR = {Tien, Pham Trong and Khoi, Le Hai},
	TITLE = {Differences of weighted composition operators between the
	{F}ock spaces},
	JOURNAL = {Monatsh. Math.},
	FJOURNAL = {Monatshefte f\"ur Mathematik},
	VOLUME = {188},
	YEAR = {2019},
	NUMBER = {1},
	PAGES = {183--193},
	ISSN = {0026-9255,1436-5081},
	MRCLASS = {47B33 (46E15)},
	MRNUMBER = {3895397},
	MRREVIEWER = {Rita\ A.\ Hibschweiler},
	DOI = {10.1007/s00605-018-1179-6},
	URL = {https://doi.org/10.1007/s00605-018-1179-6},
}

@article {MR2957216,
	AUTHOR = {Constantin, Olivia},
	TITLE = {A {V}olterra-type integration operator on {F}ock spaces},
	JOURNAL = {Proc. Amer. Math. Soc.},
	FJOURNAL = {Proceedings of the American Mathematical Society},
	VOLUME = {140},
	YEAR = {2012},
	NUMBER = {12},
	PAGES = {4247--4257},
	ISSN = {0002-9939,1088-6826},
	MRCLASS = {30H20 (47B38)},
	MRNUMBER = {2957216},
	MRREVIEWER = {Sei-ichiro\ Ueki},
	DOI = {10.1090/S0002-9939-2012-11541-2},
	URL = {https://doi.org/10.1090/S0002-9939-2012-11541-2},
}

@article {MR3192295,
	AUTHOR = {Mengestie, Tesfa},
	TITLE = {Product of {V}olterra type integral and composition operators
	on weighted {F}ock spaces},
	JOURNAL = {J. Geom. Anal.},
	FJOURNAL = {Journal of Geometric Analysis},
	VOLUME = {24},
	YEAR = {2014},
	NUMBER = {2},
	PAGES = {740--755},
	ISSN = {1050-6926,1559-002X},
	MRCLASS = {30H20 (46E22)},
	MRNUMBER = {3192295},
	MRREVIEWER = {Jordi\ Pau},
	DOI = {10.1007/s12220-012-9353-x},
	URL = {https://doi.org/10.1007/s12220-012-9353-x},
}

@article{lin,
    author       = {Lin, Zhaopeng and Lu, Yufeng and Zu, Chao},
    title        = {Toeplitz Operators on Quaternionic Fock Spaces},
    year         = {2026},
    eprint       = {2601.10162},
    archivePrefix= {arXiv},
    primaryClass = {math.FA},
    note         = {arXiv preprint, version v3, last revised 4 Jun 2026},
    url          = {https://doi.org/10.48550/arXiv.2601.10162}
}

@article {MR2819157,
	AUTHOR = {Hu, Zhangjian and Lv, Xiaofen},
	TITLE = {Toeplitz operators from one {F}ock space to another},
	JOURNAL = {Integral Equations Operator Theory},
	FJOURNAL = {Integral Equations and Operator Theory},
	VOLUME = {70},
	YEAR = {2011},
	NUMBER = {4},
	PAGES = {541--559},
	ISSN = {0378-620X,1420-8989},
	MRCLASS = {47B35 (32A37 46E15 47B10)},
	MRNUMBER = {2819157},
	MRREVIEWER = {Edgar\ Tchoundja},
	DOI = {10.1007/s00020-011-1887-y},
	URL = {https://doi.org/10.1007/s00020-011-1887-y},
}

@incollection {MR3587897,
	AUTHOR = {Alpay, Daniel and Colombo, Fabrizio and Sabadini, Irene and
	Salomon, Guy},
	TITLE = {The {F}ock space in the slice hyperholomorphic setting},
	BOOKTITLE = {Hypercomplex analysis: new perspectives and applications},
	SERIES = {Trends Math.},
	PAGES = {43--59},
	PUBLISHER = {Birkh\"auser/Springer, Cham},
	YEAR = {2014},
	ISBN = {978-3-319-08770-2; 978-3-319-08771-9},
	MRCLASS = {30G35 (30H20)},
	MRNUMBER = {3587897},
}

@book {colombo2016entire,
	AUTHOR = {Colombo, Fabrizio and Sabadini, Irene and Struppa, Daniele C.},
	TITLE = {Entire slice regular functions},
	SERIES = {SpringerBriefs in Mathematics},
	PUBLISHER = {Springer, Cham},
	YEAR = {2016},
	PAGES = {v+118},
	ISBN = {978-3-319-49264-3; 978-3-319-49265-0},
	MRCLASS = {30G35 (30B10 30D15)},
	MRNUMBER = {3585395},
	MRREVIEWER = {Alessandro\ Perotti},
	DOI = {10.1007/978-3-319-49265-0},
	URL = {https://doi.org/10.1007/978-3-319-49265-0},
}

@book {MR4292259,
	AUTHOR = {Alpay, Daniel and Colombo, Fabrizio and Sabadini, Irene},
	TITLE = {Quaternionic de {B}ranges spaces and characteristic operator
	function},
	SERIES = {SpringerBriefs in Mathematics},
	PUBLISHER = {Springer, Cham},
	YEAR = {[2020] \copyright 2020},
	PAGES = {x+116},
	ISBN = {978-3-030-38311-4; 978-3-030-38312-1},
	MRCLASS = {47-02 (30G35 46C20 46E22 46S05 47S05)},
	MRNUMBER = {4292259},
	MRREVIEWER = {Lu\'is\ P.\ Castro},
	DOI = {10.1007/978-3-030-38312-1},
	URL = {https://doi.org/10.1007/978-3-030-38312-1},
}

@article {MR173012,
	AUTHOR = {Cullen, C. G.},
	TITLE = {An integral theorem for analytic intrinsic functions on
	quaternions},
	JOURNAL = {Duke Math. J.},
	FJOURNAL = {Duke Mathematical Journal},
	VOLUME = {32},
	YEAR = {1965},
	PAGES = {139--148},
	ISSN = {0012-7094,1547-7398},
	MRCLASS = {30.83},
	MRNUMBER = {173012},
	MRREVIEWER = {R.\ F.\ Rinehart},
	URL = {http://projecteuclid.org/euclid.dmj/1077375642},
}

@article {MR3239622,
	AUTHOR = {Le, Trieu},
	TITLE = {Normal and isometric weighted composition operators on the
	{F}ock space},
	JOURNAL = {Bull. Lond. Math. Soc.},
	FJOURNAL = {Bulletin of the London Mathematical Society},
	VOLUME = {46},
	YEAR = {2014},
	NUMBER = {4},
	PAGES = {847--856},
	ISSN = {0024-6093,1469-2120},
	MRCLASS = {47B33},
	MRNUMBER = {3239622},
	MRREVIEWER = {Sei-ichiro\ Ueki},
	DOI = {10.1112/blms/bdu046},
	URL = {https://doi.org/10.1112/blms/bdu046},
}

@article {MR3801294,
	AUTHOR = {de Fabritiis, Chiara and Gentili, Graziano and Sarfatti,
	Giulia},
	TITLE = {Quaternionic {H}ardy spaces},
	JOURNAL = {Ann. Sc. Norm. Super. Pisa Cl. Sci. (5)},
	FJOURNAL = {Annali della Scuola Normale Superiore di Pisa. Classe di
	Scienze. Serie V},
	VOLUME = {18},
	YEAR = {2018},
	NUMBER = {2},
	PAGES = {697--733},
	ISSN = {0391-173X,2036-2145},
	MRCLASS = {30G35 (30H10)},
	MRNUMBER = {3801294},
}

@article {MR2353257,
	AUTHOR = {Gentili, Graziano and Struppa, Daniele C.},
	TITLE = {A new theory of regular functions of a quaternionic variable},
	JOURNAL = {Adv. Math.},
	FJOURNAL = {Advances in Mathematics},
	VOLUME = {216},
	YEAR = {2007},
	NUMBER = {1},
	PAGES = {279--301},
	ISSN = {0001-8708,1090-2082},
	MRCLASS = {30G35 (30B10 30C10)},
	MRNUMBER = {2353257},
	MRREVIEWER = {Alessandro\ Perotti},
	DOI = {10.1016/j.aim.2007.05.010},
	URL = {https://doi.org/10.1016/j.aim.2007.05.010},
}

@article {MR4162404,
	AUTHOR = {Lian, Pan and Liang, Yuxia},
	TITLE = {Weighted composition operator on quaternionic {F}ock space},
	JOURNAL = {Banach J. Math. Anal.},
	FJOURNAL = {Banach Journal of Mathematical Analysis},
	VOLUME = {15},
	YEAR = {2021},
	NUMBER = {1},
	PAGES = {Paper No. 7, 20},
	ISSN = {2662-2033,1735-8787},
	MRCLASS = {30G35 (30H20 47B33 47B38)},
	MRNUMBER = {4162404},
	MRREVIEWER = {Matthew\ M.\ Jones},
	DOI = {10.1007/s43037-020-00087-6},
	URL = {https://doi.org/10.1007/s43037-020-00087-6},
}

@article {MR4519267,
	AUTHOR = {Liang, Y. and Wang, J.},
	TITLE = {Difference of quaternionic weighted composition operators on
	slice regular {F}ock spaces},
	JOURNAL = {Complex Var. Elliptic Equ.},
	FJOURNAL = {Complex Variables and Elliptic Equations. An International
	Journal},
	VOLUME = {68},
	YEAR = {2023},
	NUMBER = {1},
	PAGES = {120--134},
	ISSN = {1747-6933,1747-6941},
	MRCLASS = {30G35 (30H20 47B38)},
	MRNUMBER = {4519267},
	DOI = {10.1080/17476933.2021.1980877},
	URL = {https://doi.org/10.1080/17476933.2021.1980877},
}

@article {MR34064751,
	AUTHOR = {Colombo, Fabrizio and Gonz\'alez-Cervantes, J. Oscar and
	Sabadini, Irene},
	TITLE = {Further properties of the {B}ergman spaces of slice regular
	functions},
	JOURNAL = {Adv. Geom.},
	FJOURNAL = {Advances in Geometry},
	VOLUME = {15},
	YEAR = {2015},
	NUMBER = {4},
	PAGES = {469--484},
	ISSN = {1615-715X,1615-7168},
	MRCLASS = {30G35 (30H20)},
	MRNUMBER = {3406475},
	MRREVIEWER = {John\ Ryan},
	DOI = {10.1515/advgeom-2015-0022},
	URL = {https://doi.org/10.1515/advgeom-2015-0022},
}

@book {MR2934601,
	AUTHOR = {Zhu, Kehe},
	TITLE = {Analysis on {F}ock spaces},
	SERIES = {Graduate Texts in Mathematics},
	VOLUME = {263},
	PUBLISHER = {Springer, New York},
	YEAR = {2012},
	PAGES = {x+344},
	ISBN = {978-1-4419-8800-3},
	MRCLASS = {30H20 (30D15 46E20 47B10 47B35)},
	MRNUMBER = {2934601},
	MRREVIEWER = {Jordi\ Pau},
	DOI = {10.1007/978-1-4419-8801-0},
	URL = {https://doi.org/10.1007/978-1-4419-8801-0},
}

@article {MR3358083,
	AUTHOR = {Arcozzi, Nicola and Sarfatti, Giulia},
	TITLE = {Invariant metrics for the quaternionic {H}ardy space},
	JOURNAL = {J. Geom. Anal.},
	FJOURNAL = {Journal of Geometric Analysis},
	VOLUME = {25},
	YEAR = {2015},
	NUMBER = {3},
	PAGES = {2028--2059},
	ISSN = {1050-6926,1559-002X},
	MRCLASS = {30G35 (46E22 58B20)},
	MRNUMBER = {3358083},
	MRREVIEWER = {Jin-xun\ Wang},
	DOI = {10.1007/s12220-014-9503-4},
	URL = {https://doi.org/10.1007/s12220-014-9503-4},
}

@book {MR3013643,
	AUTHOR = {Gentili, Graziano and Stoppato, Caterina and Struppa, Daniele
	C.},
	TITLE = {Regular functions of a quaternionic variable},
	SERIES = {Springer Monographs in Mathematics},
	PUBLISHER = {Springer, Heidelberg},
	YEAR = {2013},
	PAGES = {x+185},
	ISBN = {978-3-642-33870-0; 978-3-642-33871-7},
	MRCLASS = {30-02 (30B10 30C15 30C80 30E20 30G35)},
	MRNUMBER = {3013643},
	MRREVIEWER = {Alessandro\ Perotti},
	DOI = {10.1007/978-3-642-33871-7},
	URL = {https://doi.org/10.1007/978-3-642-33871-7},
}

	\medskip
	
	\noindent
	School of Mathematical Sciences, Dalian University of Technology,
	Dalian, Liaoning 116024, P. R. China
	
	\noindent
	Email address: \texttt{zhaopeng.lin@mail.dlut.edu.cn}
	
	\medskip
	
	\noindent
	School of Mathematical Sciences, Dalian University of Technology,
	Dalian, Liaoning 116024, P. R. China
	
	\noindent
	Email address: \texttt{lyfdlut@dlut.edu.cn}
	
	\medskip
	
	\noindent
	School of Mathematical Sciences, Dalian University of Technology,
	Dalian, Liaoning 116024, P. R. China
	
	\noindent
	Email address: \texttt{zuchao@dlut.edu.cn}

 \end{document}